\documentclass[12pt,reqno]{amsart}
\usepackage{amssymb}
\usepackage{geometry}
\usepackage{comment}
\usepackage[T1]{fontenc}
\usepackage[utf8]{inputenc}
\usepackage[italian,english]{babel}
\usepackage{amsmath, amsfonts, amsthm}
\usepackage{graphicx}
\usepackage{pgfplots}
\usepackage{paralist}
\usepackage{mathtools, mathabx,bigints}
\numberwithin{equation}{section}
\newtheorem{thm}{Theorem}[section]
\newtheorem{lem}[thm]{Lemma}
\newtheorem{cor}[thm]{Corollary}
\newtheorem{prop}[thm]{Proposition}

\newtheorem{defn}[thm]{Definition}
\theoremstyle{definition}
\newtheorem{rem}[thm]{Remark}
\theoremstyle{remark}

\newcommand{\R}{\mathbb{R}}
\newcommand{\N}{\mathbb{N}}

\newcommand{\argmin}{\arg\min}

\DeclareMathOperator{\dive}{div}

{\left\{\begin{array}{@{}l@{}}}{\end{array}\right.}
\patchcmd{\abstract}{\scshape\abstractname}{\textbf{\abstractname}}{}{}
\makeatletter 
\def\@makefnmark{} 
\makeatother 
 
\begin{document}
\title[]{The $p-$Laplace \lq\lq Signature\rq\rq \ for Quasilinear Inverse Problems with Large Boundary Data}
\author[A. Corbo Esposito, L. Faella, G. Piscitelli, R. Prakash, A. Tamburrino]{
Antonio Corbo Esposito$^1$, Luisa Faella$^1$, Gianpaolo Piscitelli$^2$, Ravi Prakash$^3$, Antonello Tamburrino$^{1,4}$}\footnote{\\$^1$Dipartimento di Ingegneria Elettrica e dell'Informazione \lq\lq M. Scarano\rq\rq, Universit\`a degli Studi di Cassino e del Lazio Meridionale, Via G. Di Biasio n. 43, 03043 Cassino (FR), Italy.\\
$^2$Dipartimento di Matematica e Applicazioni \lq\lq R. Caccioppoli\rq\rq, Universit\`a degli Studi di Napoli Federico II, Via Cinthia n. 26, Complesso Universitario Monte Sant'Angelo, 81026 Napoli, Italy.\\
$^3$Departamento de Matem\'atica, Facultad de Ciencias F\'isicas y Matem\'aticas, Universidad de Concepci\'on, Avenida Esteban Iturra s/n, Bairro Universitario, Casilla 160 C, Concepci\'on, Chile.\\
$^4$Department of Electrical and Computer Engineering, Michigan State University, East Lansing, MI-48824, USA.\\
Email: corbo@unicas.it, l.faella@unicas.it, gianpaolo.piscitelli@unina.it {\it (corresponding author)}, rprakash@udec.cl, antonello.tamburrino@unicas.it.}
\maketitle

\begin{abstract}
This paper is inspired by an imaging problem encountered in the framework of Electrical Resistance Tomography involving two different materials, one or both of which are nonlinear. Tomography with nonlinear materials is in the early stages of developments, although breakthroughs are expected in the not-too-distant future.

We consider nonlinear constitutive relationships which, at a given point in the space, present a behaviour for large arguments that is described by monomials of order $p$ and $q$.

The original contribution this work makes is that the nonlinear problem can be approximated by a {weighted} $p-$Laplace problem. From the perspective of tomography, this is a significant result because it highlights the central role played by the $p-$Laplacian in inverse problems with nonlinear materials. Moreover, when $p=2$, this provides a powerful bridge to bring all the imaging methods and algorithms developed for linear materials into the arena of problems with nonlinear materials.

The main result of this work is that for \lq\lq large\rq\rq\ Dirichlet data in the presence of two materials of different order (i) one material can be replaced by either a perfect electric conductor or a perfect electric insulator and (ii) the other material can be replaced by a material giving rise to a  {weighted}  $p-$Laplace problem.




\end{abstract}

\section{Introduction}
This paper treats a nonlinear imaging problem in Electrical Resistance Tomography (ERT). The aim is to retrieve the nonlinear electrical conductivity $\sigma$, from boundary data in stationary conditions (steady currents):
\begin{equation}\label{gproblem1}
\begin{cases}
\dive\Big(\sigma (x, |\nabla u(x)|) \nabla u (x)\Big) =0\ \text{in }\Omega\vspace{0.2cm}\\
u(x) =f(x)\qquad\qquad\qquad\quad\  \text{on }\partial\Omega,
\end{cases}
\end{equation}
where $f$ is the applied boundary potential, $u$ is the electric scalar potential and $\Omega\subset\R^n$, $n \geq 2,$ is an open bounded domain with Lipschitz boundary which represents the region occupied by the conducting material.

Retrieving the nonlinear electrical conductivity $\sigma$ from boundary data, i.e. from the $\Lambda_\sigma: f \mapsto \sigma \partial_n u$ operator, the so-called Dirichlet-to-Neumann (DtN) operator, is the nonlinear variant of the Calder\'on problem \cite{calderon1980inverse,calderon2006inverse}.

In the literature, there are very few contributions on the subject of imaging in the presence of nonlinear materials. As quoted in \cite{lam2020consistency}, {\it \lq\lq\ ... the mathematical analysis for inverse problems governed by nonlinear Maxwell's equations is still in the early stages of development.\rq\rq}. In this framework, we mention the work made for $p-$Laplace  {type}  nonlinearities \cite{Salo2012_IP,brander2015enclosure,brander2016calderon,brander2018superconductive,guo2016inverse,brander2018monotonicity,hauer2015p}, the work by Sun \cite{sun2004inverse,Sun_2005} for weak nonlinearities, the work by C\^arstea and Kar \cite{carstea2020recovery} which treated a nonlinear problem (linear plus a nonlinear term) and the work by Corbo Esposito et al. \cite{corboesposito2021monotonicity}. The latter treat a general nonlinearity within the framework of the Monotonicity Principle Method.

It can be expected that as new methods and algorithms become available, the demand for nondestructive evaluation and imaging of nonlinear materials will eventually and significantly rise.

From the physical standpoint, in steady current operations, the electric field is given by the electrical scalar potential as ${\bf E}(x)=-\nabla u(x)$, and the electrical current density ${\bf J}(x)$ depends on the electric field as follows:
\begin{equation} \label{J}
{\bf J}(x)=\sigma(x,\vert {\bf E}(x)\vert) {\bf E}(x)\quad\forall x\in\Omega.
\end{equation}
Equation \eqref{J} represents a constitutive relationship which is nonlinear, local, isotropic and memoryless.
\label{3-Applications}
 {From a general perspective, nonlinear electrical conductivities can be found in many materials like semiconducting and ceramic materials \cite{bueno2008sno2,boucher2018interest, lupo1996field}, superconductors \cite{seidel2015applied, krabbes2006high} and in biological tissues \cite{foster1989dielectric,corovic2013modeling}.
Other than steady currents, problem (\ref{gproblem1}) is common to other physical settings. Remaining in the framework of electromagnetism, both nonlinear electrostatic phenomena (see \cite{miga2011non} and references therein, see \cite{yarali20203d}) and nonlinear magnetostatic\footnote{$^1$In magnetostatics, it is possible to introduce a magnetic scalar potential for treating simply connected and source free regions.}{$^1$} phenomena (see \cite{1993ferr.book.....B}) can be modelled as in (\ref{gproblem1}).}

Specifically, we address ERT with a nonlinear electrical conductivity when the boundary data is \lq\lq large\rq\rq, \ in the sense that the boundary data is $\lambda f(x)$ where constant $\lambda$ approaches infinity. This contribution is companion of \cite{corboesposito2023thep0laplacesignature}, where ERT with nonlinear conductivities was treated in the \lq\lq small\rq\rq \ boundary data limit ($\lambda \to 0$) and it completes the \lq\lq philosophy\rq\rq \ of limiting problems. Indeed, when the electrical conductivity is unknown, no other limiting cases can be conceived, other than the small and large boundary data cases.

During the analysis of this class of problems, we recognize the limiting behaviour of the normalized solution defined as $v^\lambda = u^\lambda / \lambda$, being $u^\lambda$ the solution of \eqref{gproblem1} when the boundary data is $\lambda f$. To be more specific, we assume that the electrical conductivity $\sigma$ has a different  {behaviour for $E \to +\infty$} in region $A \subset \subset \Omega$ and $B=\Omega \backslash A$, i.e.
\begin{align}
\label{eq:asy1}
\sigma_B(x,E) \sim \beta (x)E^{p-2} \quad
\text{for a.e.}\  x \in B,\\
\label{eq:asy2}
\sigma_A(x,E) \sim \alpha (x)E^{q-2} \quad \text{for a.e.}\  x \in A,
\end{align}
 {where $\sigma_B$ and $\sigma_A$ are the restriction of $\sigma$ to $B$ and $A$, respectively, $E = \left| \mathbf{E} \right| = \left| \nabla u \right|$, }$\alpha(x)$ and $\beta(x)$ are proper functions and $p$ and $q$ are  {the asymptotic} growth exponents (see Figure \ref{fig_1_omega}).
\begin{figure}[h!]
\centering    \includegraphics[width=0.9\textwidth]{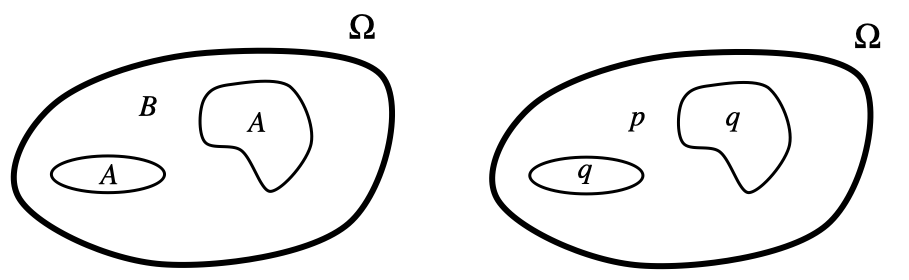}
\caption{The electrical conductivity has different  { asymptotic} growth exponents $q$ and $p$ in regions $A$ and $B$, respectively. Region $B$ surrounds $A$ and $\partial \Omega \subset \partial B$.}
    \label{fig_1_omega}
\end{figure}

In this contribution, we prove that the normalized solution $v^\lambda$ approaches a limit for $\lambda \to +\infty$ which is the solution of a  {weighted} $p-$Laplace problem in region $B$. Moreover, in solving this  {weighted} $p-$Laplace problem in region $B$, the boundary condition at the interface between $B$ and $A$ is either that of a Dirichlet boundary condition (constant scalar potential for $q>p$) or that of a Neumann boundary condition (vanishing normal component of the scalar potential for $q<p$). From the physical standpoint, region $A$ is seen as either a perfect electric conductor ($q>p$) or a perfect electric insulator ($q<p$).  {Finally, the case $p=q$ (that is the case when $A=\emptyset$) has been treated in
\cite{corboesposito2021monotonicity,MPMETHODS}}.

It is worth noting that understanding the behaviour of the solution in region $B$ is of paramount importance because the imaging system collects data, i.e. it measures a discrete approximation of $\Lambda_\sigma$, onto $\partial \Omega$, which is the outer boundary of  $B$. In other terms, the behaviour in $B$ determines the measured data.

These results make it possible to approximate a nonlinear problem with a  {weighted $p-$Laplace one in region $B$, together with a proper boundary condition on $\partial A$, corresponding to either a perfect electric conductor or a perfect electric insulator in $A$
(see Section \ref{subsec_hyp} for the underlying assumptions). They are quite interesting from both a mathematical and an engineering point of view because
they reduce the original quasilinear problem in $\Omega$ to the canonical problem of the weighted $p-$Laplace equation only in region $B$.}

 {From the perspective of inverse problems, these results allow to apply to \eqref{gproblem1} the results developed for the nonlinear $p-$Laplace variant of the Calder\'on problem, initially posed by Salo and Zhong \cite{Salo2012_IP}, and studied in \cite{brander2015enclosure,brander2016calderon,brander2018superconductive,guo2016inverse,brander2018monotonicity,hauer2015p}.
Moreover, the case $p=2$} is of paramount importance because the problem in region $B$ reduces to a linear one and, therefore, we have a powerful bridge to apply all the imaging methods and algorithms developed for linear materials to nonlinear ones.

\begin{figure}
\centering
\includegraphics[width=\textwidth]{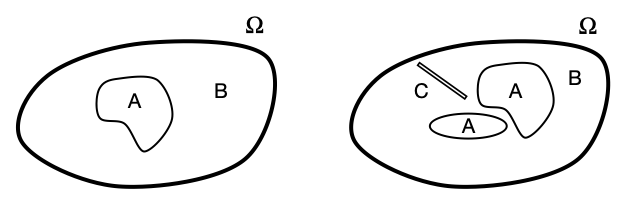}
\caption{Description of two possible applications. Left: inverse obstacle problem where the interface ($\partial A $) between two phases is unknown. $A$ and $B$ are the regions occupied by the inner material and the outer material, respectively. Right: nondestructive testing where  {regions A and B are known, while the position and shape of region $C$ (a crack) is unknown. Materials in regions $A$ and $B$ are known, as well.}}
    \label{fig_2_intro}
\end{figure}

\label{5a-IT}
 {For instance, a} classical  {class} for solving an inverse problem in the presence of linear materials  {is that of iterative methods, which is based on the (iterative) minimization} of a proper cost function \cite{tikhonov1977solutions,tikhonov1998nonlinear}. An overview of iterative methods can be found in several specialized textbooks \cite{bertero1998introduction,tarantola2005inverse,vogel2002computational,engl1996regularization}. Other than the popular Gauss-Newton method and its variants (see \cite{qi2000iteratively} for a review), we mention some relevant iterative approaches applied to inverse problems such as the Quadratic Born approximation \cite{pierri1997local}, Bayesian approaches \cite{premel2002eddy}, the Total Variation regularization \cite{RudinOsher1992nonlinear,pirani2008multi}, the Levenberg-Marquardt method for nonlinear inverse problems \cite{Hanke_1997}, the Level Set method \cite{dorn2000shape,harabetian1998regularization}, the Topological Derivative method \cite{Jackowska-Strumillo2002231,ammari2012stability,fernandez2019noniterative} and the Communication Theory approach \cite{tamburrino2000communications}.
 {Inverse problems in the presence of hollow domains, as it is the case for $B$, has been deeply investigated. Recent results on this subject includes approaches based on a parameterization of the obstacle \cite{colton1998inverse}, on the level sets method \cite{ameur2004level}, on the concept of shape derivative combined with that of level sets \cite{allaire2004structural}, on the topological derivative method \cite{bonnet2008inverse}, on the quasi-reversibility method \cite{bourgeois2010aquasi}, on a shape optimization method \cite{caubet2013shape}, and on an integral equation method \cite{kress2005nonlinear, cakoni2012integral}.} 

Noniterative methods are an excellent alternative to iterative ones, because they call for the computation of a proper function of the space, the indicator function, giving the shape of the interface between two different materials. The computation of the indicator function is much less expensive than the computation required by an iterative method, thus making noniterative methods suitable for real-time operations. Only a handful of noniterative methods are currently available. These include the Linear Sampling Method (LSM) by Colton and Kirsch \cite{Colton_1996}, which evolved into the Factorization Method (FM) proposed by Kirsch \cite{Kirsch_1998}. Ikehata proposed the Enclosure Method (EM) \cite{ikehata1999draw,Ikehata_2000} and Devaney applied MUSIC (MUltiple SIgnal
Classification), a well-known algorithm in signal processing, as an imaging method \cite{Devaney2000}. Finally, Tamburrino and Rubinacci proposed the Monotonicity Principle Method (MPM) \cite{Tamburrino_2002}.

 {Assumptions made in this work are presented in Section \ref{subsec_hyp}. They refer to a quite standard framework, apart from (A4) and (A4') which correspond to \eqref{eq:asy1} and \eqref{eq:asy2}. These two latter specific assumptions prescribe the pointwise convergence of the constitutive relationship for large electric fields $E$. Moreover, they are sharp, as shown by means of the counterexamples of Section \ref{counter_sec}. }

\label{8-arch}
The paper is organized as follows: in Section \ref{underlying} we present the ideas underpinning the work; in Section \ref{fram_sec} we set the notations and the problem, together with the required assumptions; in Section \ref{mean_sec} we study the fundamental inequality for large Dirichlet data; in Section \ref{large_sec} we discuss the limiting case for large Dirichlet data; in Section \ref{counter_sec} we provide the counterexamples proving that the specific assumptions are sharp 
and, finally, in Section \ref{Con_sec} we provide some conclusions.

\section{Underlying ideas and expected results}
\label{underlying}
In this section we present the main ideas underpinning this work. The key is the \lq\lq educated guess\rq\rq \ that when the boundary data is \lq\lq large\rq\rq, we expect the electric field $\mathbf{E}=-\nabla{u}$ to be large a.e. in $\Omega$ and, therefore, its behaviour is expected to be governed by the asymptotic behaviour of the constitutive relationship $\sigma=\sigma \left(x,E\right)$ in \eqref{gproblem1}. Let $A\subset\subset\Omega$ and  $B:=\Omega\setminus\overline A$, we assume that there exist two constants $q$ and $p$, and two functions $\beta$ and $\alpha$ which capture the behaviour of $\sigma$, as $E \to +\infty$, in $B$ and $A$, 
according to \eqref{eq:asy1} and \eqref{eq:asy2}, respectively.

Analysis of nonlinear problems is fascinating because of the wide variety of different cases. Some representative cases are shown in Figure \ref{fig_3_sigma}.
\begin{figure}
    \centering
    \includegraphics[width=0.325\textwidth]{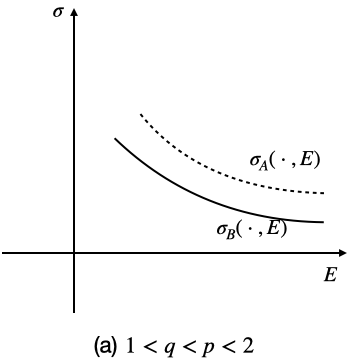}
    \includegraphics[width=0.325\textwidth]{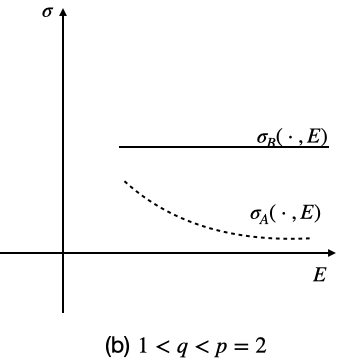}
    \includegraphics[width=0.325\textwidth]{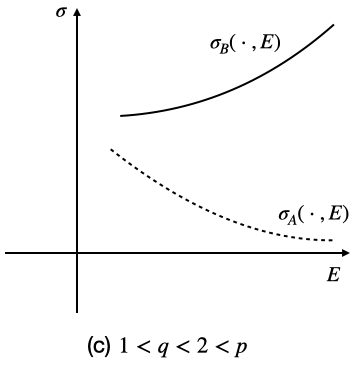}    
    \includegraphics[width=0.325\textwidth]{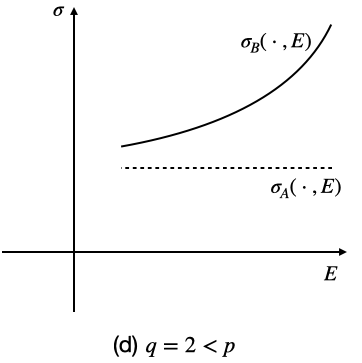}
    \includegraphics[width=0.325\textwidth]{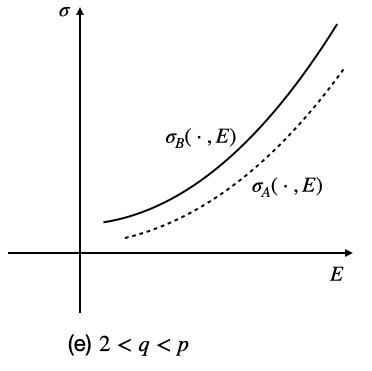}
    \caption{Relationship between the electrical conductivities for regions $A$ and $B$, when considering their asymptotic approximations: $\sigma_B(\cdot,E) \sim \beta(\cdot)E^{p-2}$ (solid line) and $\sigma_A(\cdot,E) \sim
    \alpha(\cdot) E^{q-2}$ (dashed line). Similar configurations are obtained when the order relation between $p$ and $q$ is reversed. Only the behaviour for large $E$ is shown.}
    \label{fig_3_sigma}
\end{figure}

When, for instance, $q<p$ it can be reasonably expected that either (i) region $A$ is a perfect electric insulator or (ii) region $B$ is a perfect electric conductor, because $\sigma_B$ would be dominant if compared to $\sigma_A$.
When $\partial \Omega$ is contained in $\partial B$, that is $A\subset\subset\Omega$, the ambiguity between (i) and (ii) can be resolved by looking at the boundary data. Specifically, if the boundary data is nonconstant, then region $B$ cannot be assimilated to a PEC because the potential $u$ would be constant and, therefore, noncompatible with the boundary data. Region $A$ can therefore reliably be assimilated to a PEI.

Moreover, the limiting problem where the conductor in region $A$ is replaced by a PEI can reliably be modelled by a  {weighted}  $p-$Laplace problem in region $B$, with a boundary condition on $\partial A$ given by a vanishing normal component of $\mathbf{J}$. In other words, $u \sim u_p$ in $B$, where $u_p$ is the solution of the  {weighted}  $p-$Laplace problem arising from the electrical conductivity $\beta(x) E^{p-2}$ in $B$.

The latter observation is also inspiring because it properly defines the concept of \lq\lq large\rq\rq \ boundary data and the limiting problem. Specifically, it is well known that the operator mapping the boundary data $f$ into the solution of a  {weighted}  $p-$Laplace problem is a homogeneous operator of degree 1, i.e. the solution corresponding to $\lambda f(x)$ is equal to $\lambda u_p(x)$, where $u_p$ is the solution corresponding to the boundary data $f$. Thus,  the term \lq\lq problem for large boundary data\rq\rq\ means \eqref{gproblem1} where the boundary data is $\lambda f$ and $\lambda \to +\infty$. 
Moreover, this suggests the need to study the convergent properties of the normalized solution $v^\lambda$, defined as the ratio $u^\lambda / \lambda$, where $u^\lambda$ is the solution of \eqref{gproblem1} corresponding to the Dirichlet data $\lambda f(x)$. Indeed, if $u^\lambda$ can be approximated by the solution of the  {weighted}  $p-$Laplace problem, then the normalized solution  {$v^\lambda$} converges in $B$, i.e. it is expected to be constant w.r.t. $\lambda$, as $\lambda$ approaches $+\infty$. We term this limit as $v^\infty$ and we expect it to be equal to $u_p$, i.e. the solution of the  {weighted}  $p-$Laplace problem with boundary data $f$.

From the formal point of view, when $q<p$, we will indeed prove that the normalized solution $v^\lambda$  {weakly converges to $v_\Omega$, which is termed $v_B$ in $B$ and $v_A$ in $A$. The function $v_B\in W^{1,p}(B)$} is the solution of a  {weighted}  $p-$Laplace problem in region $B$:
\begin{equation}\label{pproblem_B}
\begin{cases}
\dive\Big(\beta (x) |\nabla  {v_B}(x)|^{p-2}\nabla v_B (x)\Big) =0 & \text{in
 }B\vspace{0.2cm}\\
\beta (x) |\nabla  {v_B}(x)|^{p-2}\partial_\nu  {v_B}(x) =0 &  \text{on }\partial A\vspace{0.2cm}\\
 {v_B}(x) =f(x) & \text{on }\partial \Omega;\vspace{0.2cm}\\
\end{cases}
\end{equation}
the function  {$v_A\in W^{1,q}(A)$ is the} solution of a  {weighted} $q-$Laplace problem 
in region $A$:
\begin{equation}\label{qproblem_A}
\begin{cases}
\dive\Big(\alpha (x) |\nabla  {v_A}(x)|^{q-2}\nabla  {v_A} (x)\Big) =0 &\text{in }A\vspace{0.2cm}\\
 {v_A}(x) = {v_B}(x)  & \text{on }\partial A.
\end{cases}
\end{equation}

From the physical standpoint, problem \eqref{pproblem_B} corresponds to stationary currents where the electrical conductivity in $B$ is $\sigma(x,E)=\beta(x)E^{p-2}$, and region $A$ is replaced by a perfectly insulating material (PEI). Problem \eqref{qproblem_A} corresponds to stationary currents in region $A$, with boundary data given by the solution of \eqref{pproblem_B}, and electrical conductivity equal to $\sigma(x,E)=\alpha(x)E^{q-2}$.

 {It is worth noting that any solutions of problems \eqref{pproblem_B} and \eqref{qproblem_A} satisfy the minimum problems \eqref{Hinfty} and \eqref{Linfty}, respectively, described in Section \ref{large_sec}}.

Similarly, when $p<q$, it turns out that $v^\lambda$ converges to $ w\in W^{1,p}(\Omega)$ for $\lambda\to+\infty$, where  {$w$} is constant in each connected component of $A$ and is the solution of: 
\begin{equation}\label{pproblem_Bgrad}
\begin{cases}
\dive\Big(\beta (x) |\nabla  w(x)|^{p-2}\nabla  w (x)\Big) =0\ &\text{in }B\vspace{0.2cm}\\
 {|\nabla w|=0} &\text{a.e. in } A\\
\int_{\partial A}\sigma(x,|\nabla w(x)|)\partial_\nu w(x)dS=0\\
 w(x) =f(x)\qquad\qquad\qquad\qquad\quad  &\text{on }\partial \Omega\vspace{0.2cm}
\end{cases}
\end{equation}
In this case, from the physical standpoint,  region $A$ can be replaced by a Perfect Electric Conductor (PEC). 

 {It is worth noting that the solution of problem \eqref{pproblem_Bgrad} satisfies the minimum problem \eqref{N}, described in Section \ref{large_sec}}.

Finally, we highlight that in both cases ($q<p$ or $q>p$), the limiting problems in regions $B$ and $A$ can be solved in cascade (see Figure \ref{fig_4_decoupling}). Specifically, the problem for region $B$ can be solved first, then the problem for region $A$ can be solved from the knowledge of the trace on $\partial A$ of the solution evaluated in $B$.

\begin{figure}[h!]
\centering
\includegraphics[width=0.45\textwidth]{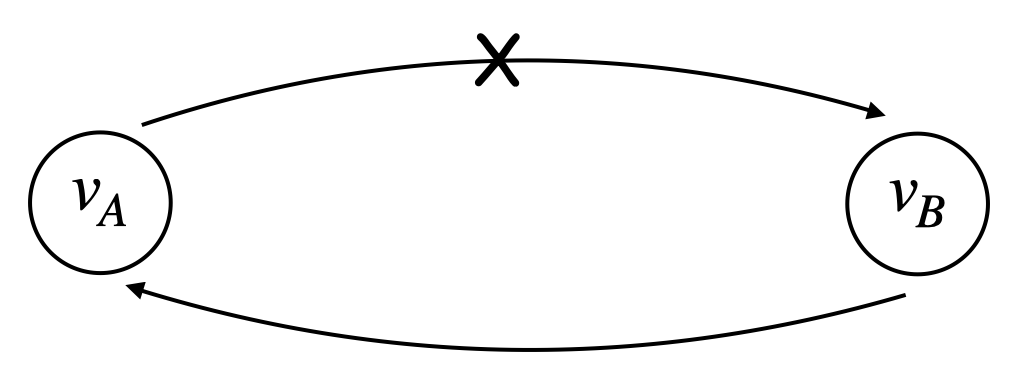}
\caption{The problem in $B$ can be solved independently of the problem in $A$. The problem in $A$ can be solved through the knowledge of the trace on $\partial A$ of the solution in $B$.}
\label{fig_4_decoupling}
\end{figure}

\section{Framework of the Problem}
\label{fram_sec}
\subsection{Notations}

Throughout this paper, $\Omega$ denotes the region occupied by the conducting materials. We assume that $\Omega\subset\R^n$, $n\geq 2$, is a bounded domain (i.e. an open and connected set) with Lipschitz boundary and $A\subset\subset\Omega$ is an open bounded set with Lipschitz boundary and a finite number of connected components, such that $B:=\Omega\setminus\overline A$ is still a domain.
Hereafter we consider  {$1<p,q<+\infty$, $p\neq q$}. Region $B$ is occupied by a conducting material with an  {asymptotic} $p-$growth$^*$\footnote{$^*$  {The concept of  {asymptotic} $p-$growth and  {asymptotic} $q-$growth is formalized in assumption (A4) and (A4'), respectively.} 
} (for large electric fields) whereas region $A$ is occupied by the material with an  {asymptotic} $q-$growth (for large electric fields), see Figure \ref{fig_1_omega}.

We denote 
by $dS$ the $(n-1)-$dimensional Hausdorff measure. Moreover, we set
\[
L^\infty_+(\Omega):=\{\theta\in L^\infty(\Omega)\ |\ \theta\geq c_0\ \text{a.e. in}\ \Omega, \ \text{for a positive constant}\ c_0\}.
\]
Furthermore,  {the Sobolev space} $W^{1,p}_0(\Omega)$ is the closure set of $C_0^1(\Omega)$ with respect to the $W^{1,p}-$norm.

 The applied boundary voltage $f$ belongs to the abstract trace space $B_{p}^{1-\frac 1p,p}(\partial\Omega)$, which is a Besov space (refer to \cite{JERISON1995161,leoni17}),  {equipped with the following norm:
\[
||u||_{B^{1-\frac 1p,p}(\partial\Omega)}=||u||_{L^p(\partial\Omega)}+|u|_{B^{1-\frac 1p,p}(\partial\Omega)}<+\infty,
\]
where for the precise definition of $|\cdot|_{B^{1-\frac 1p,p}(\partial\Omega)}$ we refer to \cite[Def.18.32]{leoni17}.}

 {
For the sake of brevity, we denote this space by $X^p(\partial \Omega)$ and we point out that its elements can be identified as the functions in $W^{1,p}(\Omega)$, modulo the equivalence relation $f\in [g]_{X^p(\partial \Omega)}$ if and only if $f-g\in W^{1,p}_0(\Omega)$.}

Finally, we denote by $X^p_\diamond (\partial \Omega)$ the set of elements in $X^p(\partial \Omega)$ with zero average on $\partial\Omega$ with respect to the measure $dS$.

\subsection{The Scalar Potential and Dirichlet Energy}

In terms of the electric scalar potential, that is ${\bf E}(x)=-\nabla u(x)$, the nonlinear Ohm's law ${\bf J}(x)=\sigma(x,E(x)){\bf E}(x)$ is
 \begin{equation*}
 {\bf J} (x)=- \sigma (x, |\nabla u(x)|)\nabla u(x),
 \end{equation*}
where $\sigma$ is the electrical conductivity, ${\bf E}$ is the electric field, and ${\bf J}$ is the electric current density.

The electric scalar potential $u$ 
solves the steady current problem:
 \begin{equation}\label{gproblem}
\begin{cases}
\dive\Big(\sigma (x, |\nabla u(x)|) \nabla u (x)\Big) =0\ \text{in }\Omega\vspace{0.2cm}\\
u(x) =f(x)\qquad\qquad\qquad\quad\  \text{on }\partial\Omega,
\end{cases}
\end{equation}
where $f\in X_\diamond^p(\partial \Omega)$. Problem \eqref{gproblem} is meant in the weak sense, that is
\begin{equation*}
\int_{\Omega }\sigma \left( x,| \nabla u(x) |\right) \nabla u (x) \cdot\nabla \varphi (x)\ \text{d}x=0\quad\forall\varphi\in C_c^\infty(\Omega).
\end{equation*}

We observe that the solution $u$ restricted to $B$ belongs to $W^{1,p}(B)$, whereas $u$ restricted to $A$ belongs to $W^{1,q}(A)$. Therefore, the solution $u$ as a whole is an element of the larger functional space $W^{1,p}(\Omega)\cup W^{1,q}(\Omega)$. Furthermore, (i) if $p\leq q$ then $W^{1,p}(\Omega)\cup W^{1,q}(\Omega)=W^{1,p}(\Omega)$, and (ii) if $p\geq q$ then $W^{1,p}(\Omega)\cup W^{1,q}(\Omega)=W^{1,q}(\Omega)$.

The solution $u$ satisfies the boundary condition in the sense that $u-f\in W_0^{1,p}(\Omega)\cup W_0^{1,q}(\Omega)$ and we write $u|_{\partial\Omega}=f$.

Moreover, the solution $u$ is variationally characterized as
\begin{equation}\label{gminimum}
\argmin\left\{ \mathbb{F}_\sigma\left( u\right)\ :\ u\in W^{1,p}(\Omega)\cup W^{1,q}(\Omega), \ u|_{\partial\Omega}=f\right\}.
\end{equation}

In (\ref{gminimum}), the functional $\mathbb{F}_\sigma\left( u\right)$ is the Dirichlet Energy
\begin{equation*}
\mathbb{F}_\sigma
\left(  u \right) = \int_{B} Q_B (x,|\nabla u(x)|)\ \text{d}x+ \int_A Q_A (x,|\nabla u(x)|)\ \text{d}x
\end{equation*} 
where $Q_B$ and $Q_A$ are the Dirichlet Energy density in $B$ and in $A$, respectively:
\begin{align*}
& Q_{B} \left( x,E\right)  :=\int_{0}^{E} \sigma_B\left( x,\xi \right)\xi  \text{d}\xi\quad \text{for a.e.}\ x\in B\ \text{and}\ \forall E\geq0,\\
& Q_{A}\left( x,E\right)  :=\int_{0}^{E} \sigma_A\left( x,\xi \right)\xi  \text{d}\xi\quad \text{for a.e.}\ x\in A\ \text{and}\ \forall E\geq 0,
\end{align*}
and $\sigma_B$ and $\sigma_A$ are the restriction of the electrical conductivity $\sigma$ in $B$ and $A$.

\subsection{Requirements on the Dirichlet Energy densities}\label{subsec_hyp}
In this Section, we provide the assumptions on the Dirichlet Energy densities $Q_B$ and $Q_A$, to guarantee the well-posedness of the problem and to prove the main convergence results of this paper.

For each individual result, we will make use of a  {suitable subset} of  assumptions, among those listed in the following.

Firstly, we recall the definition of the Carathéodory function.
\begin{defn}
$Q:\Omega\times[0,+\infty)\to\R$ is a Carathéodory function iff:
\begin{itemize}
\item $x\in\overline\Omega\mapsto Q(x,E)$ is measurable for every $E\in[0,+\infty)$,
\item $E\in [0,+\infty)\mapsto Q(x, E)$ is continuous for almost every $x\in\Omega$.
\end{itemize}
\end{defn}
The existence and the uniqueness of the solution of \eqref{gminimum} are guaranteed by the following assumptions on $Q_B$ and $Q_A$.
\begin{itemize}
\item[\textbf{({A1})}]  $Q_B$ and $Q_A$ are Carathéodory functions;
\item[\textbf{({A2})}] $E\mapsto Q_B(x,E)$ and $E\to Q_A(x,E)$ are positive, $C^1$,  strictly convex function such that $Q_B(x,0)=0$ for a.e. $x\in B$, and $Q_A(x,0)=0$ for a.e. $x\in A$.
\end{itemize}
Since $Q_B$ and $Q_A$ are positive, convex and vanishing in $0$, they are both increasing functions in $E$.
\begin{itemize}
\item[{\bf ({A3})}] There exist  {three} positive constants $\underline Q\leq \overline Q$ and $E_0$, such that: 
\[
\begin{split}
(i)\ \underline Q \left[\left(\frac{E}{E_0}\right)^p- 1\right]\leq Q_B(x, E)\leq \overline Q \left[\left(\frac{E}{E_0}\right)^p+ 1\right]\ \ \text{for a.e.} \  x\in B \ \text{and}\ \forall\  E\ge 0,\\
(ii)\ \underline Q \left[\left(\frac {E}{E_0}\right)^q- 1\right]\leq Q_A(x, E)\leq \overline Q \left[\left(\frac{E}{E_0}\right)^q+ 1\right]\ \ \text{for a.e.} \  x\in A \ \text{and}\ \forall\  E\ge 0.
\end{split}
\]
\end{itemize}


 {Assumptions (A1), (A2) and (A3) assure the existence of the solution of problem \eqref{gminimum}, see Theorem \ref{existence_thm} below}.

Assumption ({A3}) is well-known in literature (see e.g. assumptions (H4) in \cite{corboesposito2021monotonicity} and (A2) in \cite{lam2020consistency}).

In order to obtain the convergence results for large Dirichlet boundary data, we  {introduce the concept of  {asymptotic} $p-$growth and  {asymptotic} $q-$growth, as in the following assumption.}
\begin{itemize}
\item[{\bf ({A4})}] There exists a function $\beta$ such that:
\begin{equation*}
\begin{split}
\lim_{E\to +\infty} \frac{Q_B (x,E)}{E^p}=\beta(x)\quad \text{for a.e.}\  x\in B.
\end{split}
\end{equation*}
\end{itemize}
We observe that $\beta\in L^\infty_+(B)$, by (A3).  {This is a key assumption for this contribution: it guarantees the convergence of the solution in region $B$.} In Section \ref{counter_sec}, we provide a counterexample to show that assumption (A4) is sharp. 

\begin{itemize}
\item[{\bf ({A4}')}] There exists a function $\alpha$ such that:
\begin{equation*}
\begin{split}
\lim_{E\to +\infty} \frac{Q_A (x,E)}{E^q}=\alpha(x)\quad \text{for a.e.}\  x\in A.
\end{split}
\end{equation*}
\end{itemize}
We observe that $\alpha\in L^\infty_+(A)$ by ({{A3}}).  {This assumption is needed to guarantee the convergence of the solution in region $A$ for $p>q$. Assumption (A4') is sharp.}

\subsection{Connection among $\sigma$, ${\bf J}$ and $Q$}

This paper is focused on the properties of the Dirichlet Energy density $Q$, while, in physics and engineering the electrical conductivity $\sigma$ is of greater interest. From this perspective, assumptions ( {Ax}) are able to include a wide class of electrical conductivities (see Figure \ref{fig_5_assumptions}). 
In other words, the ( {Ax})s are not restrictive in practical applications.

\begin{figure}[ht]
	\centering
	\includegraphics[width=0.42\textwidth]{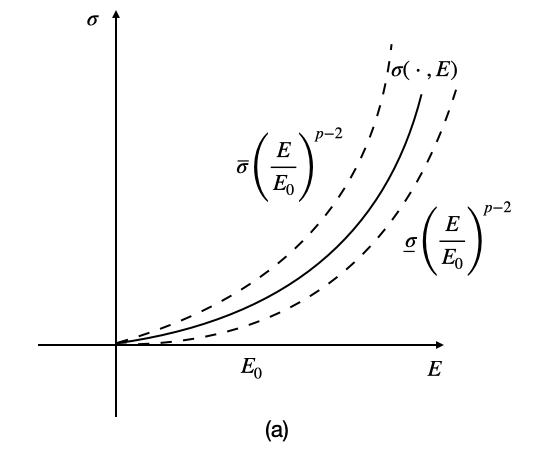}
	\includegraphics[width=0.42\textwidth]{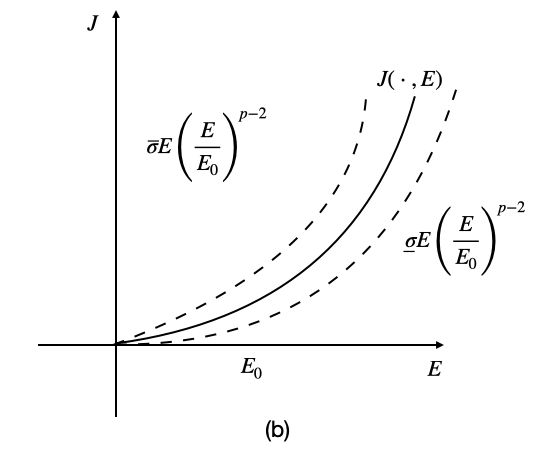}
	\includegraphics[width=0.42\textwidth]{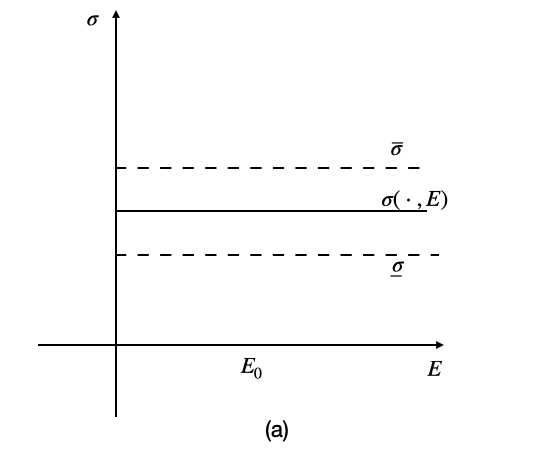}
	\includegraphics[width=0.42\textwidth]{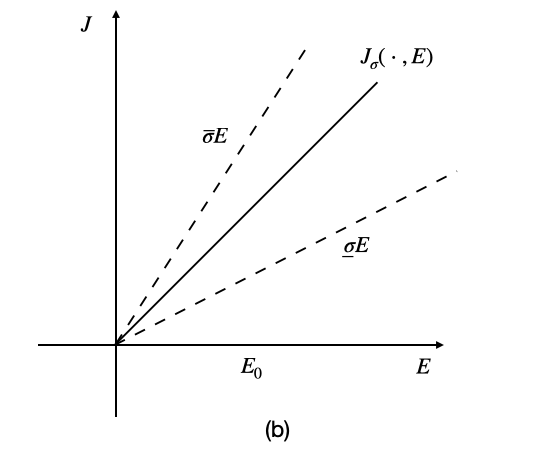}
	\includegraphics[width=0.42\textwidth]{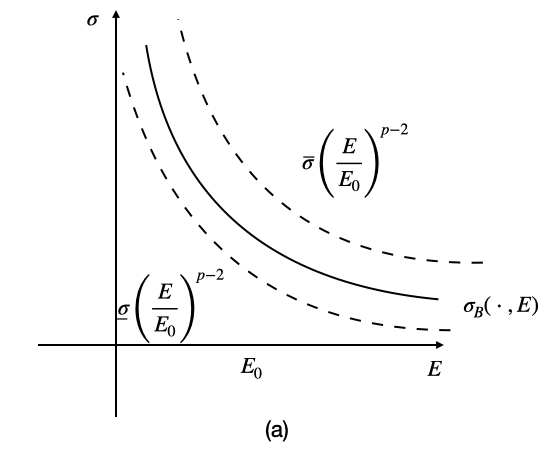}
	\includegraphics[width=0.42\textwidth]{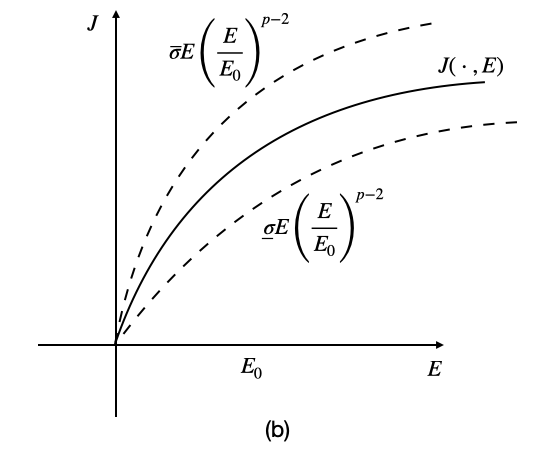}
	\caption{Example of a constitutive relationship when $p>2$, $p=2$ and $p<2$: (a) in terms of electrical conductivity and (b) in terms of current density.  {Dashed} lines correspond to the upper and lower constraints to either $\sigma$ or ${J}_\sigma$. $\underline \sigma$ $(\overline \sigma)$ is related to $\underline Q$ $(\overline Q)$, $p$ and $E_0$.}
	\label{fig_5_assumptions}
\end{figure}

There is a close connection between $\sigma$, $J_\sigma$ and $Q_\sigma$. Indeed, we observe that 
\begin{equation*}
\begin{split}
Q_B
\left( x,E \right)  =\int_{0}^{E}{J_B} ({ x, \xi}) \ \text{d}{ \xi}\quad\text{for a.e.} \ x\in B\ \text{and}\ \forall E> 0,\\
Q_A
\left( x,E \right)  =\int_{0}^{E}{J_A} ({ x, \xi}) \ \text{d}{ \xi}\quad\text{for a.e.} \ x\in A\ \text{and}\ \forall E> 0,
\end{split}
\end{equation*}
where $J_B$ and $J_A$ is the magnitude of the current density in regions $B$ and $A$, respectively:
\begin{equation}\label{connsJQ}
    \begin{split}
J_B (x, E)&=\partial_E Q_B(x,E)=\sigma_B(x, E)E\quad \text{for a.e.} \ x \in B\ \text{and}\ \forall E>0.\\
J_A (x, E)&=\partial_E Q_A(x,E)=\sigma_A(x, E)E\quad \text{for a.e.} \ x \in A\ \text{and}\ \forall E>0.
    \end{split}
\end{equation}

The electrical conductivity $\sigma(x,E)$ is the secant to the graph of the function $J_\sigma(x,E(x))$ and $Q_\sigma (x, E(x))$ is the area of the sub-graph of $J_\sigma(x, E(x))$. For a geometric interpretation of the connections between $\sigma$, $J_\sigma$ and $Q_\sigma$, see Figure \ref{fig_6_JE}.

\begin{figure}[!ht]
	\centering
	\includegraphics[width=0.45\textwidth]{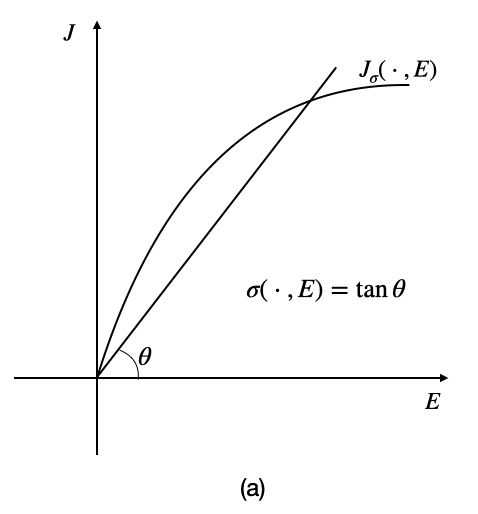}
	\includegraphics[width=0.45\textwidth]{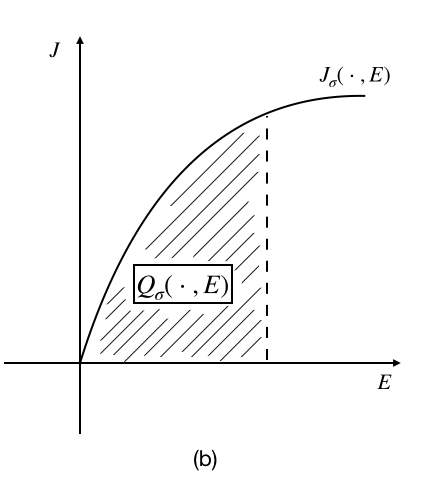}
	\caption{For any given spatial point in the region $\Omega$, (a) the electrical conductivity $\sigma(\cdot,E)$ is the secant line to the graph of the function $J_\sigma(\cdot,E)$;
	(b) $Q_\sigma (\cdot, E)$ is the area of the sub-graph of $J_\sigma(\cdot, E)$.}
	\label{fig_6_JE}
\end{figure}

\subsection{Existence and uniqueness of the solutions}
The proof of the existence and uniqueness of the solution for problem \eqref{gminimum}, relies on standard methods of the Calculus of Variations, when the Dirichlet Energy density presents the same growth  {bounds (described in (A3))} in any point of the domain $\Omega$. The case treated in this work is nonstandard, because the Dirichlet Energy density presents different growth  {bounds} in $B$ and $A$ and, hence, we provide a proof in the following.

\begin{thm}\label{existence_thm}
Let  {$1<p, q<+\infty$, $p\neq q$}, $f\in X_\diamond^p(\partial \Omega)$. If (A1), ({A2}), (A3) hold, then there exists a unique solution of problem \eqref{gminimum}.
\end{thm}
\begin{proof}
 {Let us consider the case $p>q$; if $E\ge E_0$ then  $ E^p\ge E_0^{p-q} E^q$}. By assumption (A3),  {for any $E\ge E_0$,} we have
\[
\begin{split}
& Q_B(x, E)\ge \underline{Q}\left[ \left(\frac{E}{E_0}\right)^p {-1}\right]\ge \underline{Q}\left[\left(\frac{E}{E_0}\right)^q  {-1}\right]\quad\text{a.e. in } B,\\
& Q_A(x, E)\ge \underline{Q}\left[ \left(\frac{E}{E_0}\right)^q {-1}\right]\qquad\qquad\qquad\qquad\quad\text{a.e. in } A.
\end{split}
\]
Therefore, setting $Q_\sigma=Q_B$ in $B$ and $Q_\sigma=Q_A$ in $A$, we have
\[
Q_\sigma(x, E)\ge  \underline{Q}\left[\left(\frac{E}{E_0}\right)^q {-1}\right]\quad \text{a.e. in } \Omega,
\]
 {$\text{for any }E\ge 0$ (since the inequality is trivial for any $E\ge E_0$).} Similarly, when $p<q$, we have
\[
Q_\sigma(x,E)\ge\underline{Q}\left[ \left(\frac{E}{E_0}\right)^p {-1}\right]\quad \text{a.e. in } \Omega,\]
for any  {$E\ge 0$}.

 {Hence, in both cases, the function $Q_\sigma(x,\cdot)$ is coercive for a.e. $x\in\Omega$. Therefore, since $Q_\sigma$ is also strictly convex by assumption, standard methods of the Calculus of Variations \cite[Th. 3.30]{dacorogna2007direct} assure that the solution exists and is unique.}
\end{proof}

Condition $f\in X_\diamond^p(\partial \Omega)$ is a  necessary assumption because it guarantees the existence of a function in $f+W_{0}^{1,p}(\Omega)$ such that $\mathbb F_\sigma (f)<+\infty$ (see \cite[Th. 3.30]{dacorogna2007direct} for details). Moreover, optimization problems on domains with holes have received a great deal of interest in recent years, see e.g. \cite{della2020optimal,gavitone2021isoperimetric,paoli2020stability,paoli2020sharp} and reference therein.

\subsection{Normalized solution}
Through this paper we study the behaviour of the solution of problem \eqref{gminimum} for large Dirichlet boundary data, i.e. the behaviour of $u^\lambda$ defined as  {the solution of}:
\begin{equation}
    \label{Fspezzata}
\min_{\substack{u\in W^{1,p}(\Omega)\cup W^{1,q}(\Omega)\\ u=\lambda f\ \text{on}\ \partial \Omega}}\mathbb F_\sigma(u),
\end{equation}
for $\lambda\to +\infty$.

To this purpose, as discussed in Section \ref{underlying}, it is convenient to introduce the normalized solution $v^\lambda$ defined as:
\begin{equation*}
v^\lambda=\frac{u^\lambda}{\lambda}.
\end{equation*} 


For any prescribed $f\in X^p_\diamond(\partial \Omega)$ and $\lambda>0$, $v^\lambda$ is the solution of the following variational problem:
\begin{equation}\label{G_norm}
\min_{\substack{v\in W^{1,p}(\Omega)\cup W^{1,q}(\Omega)\\ v=f\ \text{on}\ \partial \Omega}}\mathbb G^\lambda(v),\quad \mathbb G^\lambda(v)=\frac 1 {\lambda^p}\left(\int_{B} Q_B(x,\lambda|\nabla v(x)|)dx+\int_{A} Q_A(x,\lambda|\nabla v(x)|)dx\right).
\end{equation}
The multiplicative factor $1/\lambda^p$ is introduced in order to guarantee that the  {minimum values of the} functional $\mathbb G^\lambda$ are bounded for large $\lambda$. The normalized solution makes it possible to \lq\lq transfer\rq\rq\ parameter $\lambda$ in \eqref{Fspezzata} from the boundary data to the functional $\mathbb G^\lambda$.

Specifically, we will prove that $v^\lambda$ converges, under very mild hypotheses, for $\lambda\to+\infty$ to the solution of a problem where the material in region $A$ is replaced by either a PEC or a PEI.

 {In the PEC case ($p<q$) the solution $w$ is constant in the region A.}

 {In the PEI case ($p>q$)  the limit of $v^\lambda$ in $B$ is termed $v_B$, whereas the limit of $v^\lambda$ in $A$ is termed $v_A$.
Moreover (i) $v_B$ and $v_A$ arise from a  {weighted}  $p-$Laplace and a  {weighted}  $q-$Laplace problem, respectively and (ii) $v_B$ provides the boundary data on $\partial A$ for evaluating $v_A$.}


\section{The fundamental inequality for large Dirichlet data}
\label{mean_sec}
In this Section we provide the main tool to achieve the convergence results in the limiting cases for \lq\lq large\rq\rq\ Dirichlet boundary data. Specifically, we will show that the asymptotic behaviour of the Dirichlet Energy corresponds to a $p-$Laplace modelled equation \cite{Salo2012_IP,guo2016inverse} in the domain $B$.

In the following, we study the asymptotic behaviour of the Dirichlet Energy in the outer region $B$.

We first analyse the special case when the Dirichlet Energy density $Q_B$ is modelled on the  {weighted}  $p-$Laplace case, that is $Q_B(x,E)=\theta(x)E^p$ on $B\times[0,+\infty)$; we then give the result in the more general case. 

\begin{lem}
\label{factorizable_lemma3}
Let $1<
p<+\infty$, $F
\subset\R^n$ be a bounded domain with Lipschitz boundary, 
 {$\theta\in L^\infty_+(F)$} and 
$\{ {h}_n\}_{n\in\N}$ be 
a sequence weakly convergent  to $ {h}$ in $W^{1,p}(F)$. 
Then we have  \begin{equation}
\label{only_in_factor_E3}
\int_{F} \theta(x)|\nabla  {h}(x)|^p dx\le\liminf_{n\to+\infty}\int_{F} \theta(x)|\nabla  {h}_n(x)|^p dx.
\end{equation}
\end{lem}
\begin{proof}
Let us set $L:=\liminf_{n\to+\infty}\int_{F} \theta(x)|\nabla  {h}_n(x)|^p dx$.  {If $L=+\infty$, the inequality \eqref{only_in_factor_E3} is trivial; otherwise we} consider a subsequence  {$\{n_j\}_{j\in\N}$} such that
\[
\lim_{ {j}\to+\infty}\int_{F} \theta(x)|\nabla  {h_{n_j}}(x)|^p dx=L.
\]
This means that for any $\varepsilon>0$, there exists $\nu\in\N$ such that
\begin{equation}
    \label{defin_liminf}
L-\varepsilon<\int_{F} \theta(x)|\nabla  {h_{n_j}}(x)|^p dx<L+\varepsilon
\end{equation}
for any $ {j}\ge\nu$. 
Then, by Mazur's Lemma (refer for example to
\cite{dunford1963linear,renardy2006introduction}), for any $n\in\N$ there exists a function $N:\N\to\N$ and a sequence $\{\alpha_{n, {l}}\}_{ {l}=n}^{N(n)}$ such that
\begin{enumerate}
\item[(M1)] $\alpha_{ {n},l}\geq 0$\ for any $ {l}\in[n,N(n)]$, 
\item[(M2)] $\sum_{ {l}=n}^{N(n)}\alpha_{n, {l}}=1$,
\item[(M3)] $z_{ {n}}:=\sum_{ {l}=n}^{N(n)}\alpha_{n, {l}} {h_{n_l}}$ $\to  {h}$ in $W^{1,p}(F)$.
\end{enumerate}
Then there exists a subsequence $\{z_{n_{ {k}}}\}_{ {k}\in\N}$ such that
\begin{equation}
\label{eguagl_liminf}
\liminf_{n\to+\infty} \int_{F} \theta(x)|\nabla z_n(x)|^p dx=
\lim_{ {k}\to+\infty}\int_{F} \theta(x)|\nabla z_{n_{ {k}}}(x)|^p dx
\end{equation}
and another subsequence, still indicated with $\{z_{n_{ {k}}}\}_{ {k}\in\N}$,
such that $\nabla z_{n_{ {k}}}\to \nabla  {h}$ a.e. in  {$F$} \cite[Chap. 18]{leoni17}. Therefore, we have
\begin{equation*}
    \begin{split}
\mathbb \int_{F} \theta(x)|\nabla  {h}(x)|^p dx&=  \int_{F} \theta(x)\lim_{ k\to+\infty}|\nabla z_{n_k}(x)|^p dx\leq
\liminf_{ {k}\to+\infty}\int_{F} \theta(x)|\nabla z_{n_{ {k}}}(x)|^p dx\\
&=\lim_{ {k}\to+\infty}\int_{F} \theta(x)|\nabla z_{n_{ {k}}}(x)|^p dx =\liminf_{n\to+\infty}\int_{F} \theta(x)|\nabla z_n(x)|^p dx\\
&\le\liminf_{n\to+\infty}\sum_{ {l}=n}^{N(n)}\alpha_{n, {l}}\int_{F} \theta(x)|\nabla  {h_{n_l}}(x)|^p dx\\
&<\liminf_{n\to+\infty}\mathbb   \sum_{ {l}=n}^{N(n)}\alpha_{n, {l}}\left(L+\varepsilon\right)=L+\varepsilon\\
&=\liminf_{n\to+\infty}\int_{F} \theta(x)|\nabla  {h}_n(x)|^p dx+\varepsilon,
\end{split}
\end{equation*}
where in the first line the equality follows from the convergence result of (M3) and the inequality follows from Fatou's Lemma, in the second line we applied \eqref{eguagl_liminf}, in the third line we applied the convexity of $|\cdot|^p$, and in the fourth line we applied \eqref{defin_liminf}.

Conclusion \eqref{only_in_factor_E3} follows from the arbitrariness of $\varepsilon>0$.
\end{proof}
When $Q_B$ is factorized, the assumptions (A4) are satisfied.

The next step consists in extending \eqref{only_in_factor_E3} from a  {weighted}  $p-$Laplace case to the quasilinear case. In doing this, we restrict the validity of the result to sequences of the solutions of problem \eqref{G_norm}.

The main difficulty in proving this result lies in evaluating an upper bound of the Lebesgue measure of that part of $B$ where the solutions $v^\lambda$ admit large values of the gradient (see Figure \ref{fig_7_insiemi}).
\begin{figure}
\centering
\includegraphics[width=\textwidth]{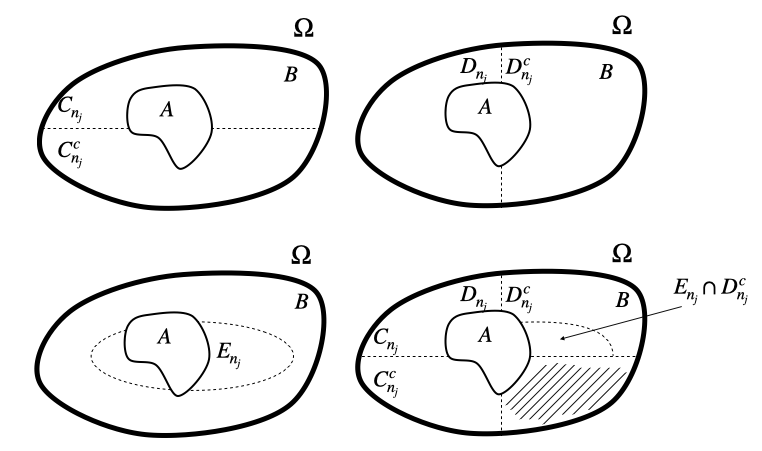}
\caption{The objective of the proof is to show that the set of the points in $B$ such that the solution $v^{\lambda_{n_j}}$ does not satisfy the fundamental inequality ($C^c_{n_j}$) and admits large values of the gradient ($D^c_{n_j}$) can be made small enough (shaded region).}
    \label{fig_7_insiemi}
\end{figure}
\begin{lem}
\label{lem_dis_fund_infty}
Let $1<p<+\infty$, $f\in X^p_\diamond(\partial \Omega)$, (A1), (A2), (A3), (A4) hold and let the solutions $v^\lambda$ of \eqref{G_norm} be weakly convergent to $v$ in $W^{1,p}(B)$, as $\lambda\to+\infty$. Let $\{\lambda_n\}_{n\in\N}$ be an increasing sequence such that $\lambda_n\to+\infty$, as $n\to+\infty$, and
\begin{equation}
    \label{seq_to_liminf_lem_infty}
\lim_{n\to+\infty}\frac{1}{\lambda_n^p}\int_{B}Q_B(x,\lambda_n|\nabla v^{\lambda_n}(x)|)dx=\liminf_{\lambda\to+\infty}\frac{1}{\lambda^p}\int_{B}Q_B(x,\lambda|\nabla v^{\lambda}(x)|)dx.
\end{equation}
Then, for any $0<\delta\le\inf_{B}\beta$ and $\theta>0$, there exists $F_{\delta,\theta}\subseteq B$ with $|B\setminus F_{\delta,\theta}|<\theta$ such that
\begin{equation}
\label{step_to_conclude_infty}
\liminf_{n\to+ \infty}\int_{F_{\delta,\theta}} (\beta(x)-\delta)|\nabla v^{\lambda_n}(x)|^p dx\le \liminf_{\lambda\to+\infty}\frac{1}{\lambda^p}\int_{B}Q_B(x,\lambda|\nabla v^{\lambda}(x)|)dx,
\end{equation}
where $\beta$ is given in (A4).
\end{lem}
\begin{proof}
Let us fix $0<\delta\le\inf_{B}\beta$ and $\theta>0$, for any $n\in\N$, we set
\[
C_{\delta,n}=C_{n}:=\left\{x\in B\ : \ (\beta(x)-\delta)(\lambda_n|\nabla v^{\lambda_n}(x)|)^p\le Q_B(x,\lambda_n|\nabla v^{\lambda_n}(x)|)\right\}.
\]
Now, for any constant $\ell>0$, we set
\begin{equation*}
\begin{split}
D_{\ell,n}=D_{n}:=\{ x\in B \ : \ \lambda_n|\nabla v^{\lambda_n}(x)|\le \lambda_n\ell\}\\
D_{\ell,n}^c=D^c_{n}:=\{ x\in B \ : \ \lambda_n|\nabla v^{\lambda_n}(x)|>\lambda_n \ell\}
\end{split}
\end{equation*}
Now, let us define
\[
E_{\delta,\ell,n}=E_{n}:=\left\{x\in B\ : \ (\beta(x)-\delta)E^p\leq Q_B(x,E)\ \  \forall E> \lambda_n \ell\right\}.
\]
Therefore, we have
\begin{equation}
\label{chain_inclusion3}
\begin{split}
& C_{n}\supseteq C_{n}\cap D^c_{n}\supseteq E_{n}\cap D^c_{n},\\
& C_{n}\cup D_{n}\supseteq
E_n\cup D_n
=(D^c_{n}\setminus E_{n})^c,
\end{split}
\end{equation}
for any $n\in\N$.
Let us observe that $E_{n}$ is increasing with respect to $n$ and that $\left| \bigcup_{n=1}^{+\infty}E_{n}\right|=|B|$. Therefore there exists a natural number $n_1=n_1(\theta)$ such that
\begin{equation}
    \label{stimaF3}
\left|\bigcup_{n=1}^{n_1} E_{n}\right|=\left|E_{n_1}\right|\ge |B|-\frac \theta 2.
\end{equation}
Passing to the complementary sets in \eqref{chain_inclusion3}, we obtain
\[
C^c_{n_1}\cap D^c_{n_1}\subseteq  D^c_{n_1}\setminus E_{n_1}\subseteq B \setminus E_{n_1}\quad\text{with}\quad |C_{n_1}^c\cap D^c_{n_1}|\leq 
\frac\theta 2,
\]
where we have used 
\eqref{stimaF3}. Analogously, we construct a subsequence $\{\lambda_{n_j}\}_{j\in\N}$ such that
\[
|C_{n_j}^c \cap D^c_{n_j}|\leq \frac{\theta}{2^j}.
\]
Then, by defining
\[
F_{\delta,\theta}: =\bigcap_{j=1}^\infty (C_{n_j}\cup D_{n_j}), 
\]
we have
\[
F^c_{\delta,\theta}= \bigcup_{j=1}^\infty (C_{n_j} \cup D_{n_j})^c
\quad\text{with}\quad
|F^c_{\delta,\theta}|\leq \sum_{j=1}^{+\infty} \frac{\theta}{2^j}=\theta,
\]
which implies
\[
|B |-\theta\le |F_{\delta,\theta}|\le |B|.
\]
Therefore, we have
\begin{equation*}
\begin{split}
&\liminf_{n\to+\infty}\int_{F_{\delta,\theta}}(\beta(x)-\delta)|\nabla v^{\lambda_{n}}(x)|^p dx\leq\liminf_{j\to+\infty}\int_{F_{\delta,\theta}}(\beta(x)-\delta)|\nabla v^{\lambda_{n_j}}(x)|^p dx\\
&\le\liminf_{j\to+\infty}\int_{C_{n_j}\cup D_{n_j}}(\beta(x)-\delta)|\nabla v^{\lambda_{n_j}}(x)|^p dx\\
&\le\liminf_{j\to+\infty}\int_{C_{n_j} }(\beta(x)-\delta)|\nabla v^{\lambda_{n_j}}(x)|^p dx+\limsup_{j\to+\infty}\int_{D_{n_j}}(\beta(x)-\delta)|\nabla v^{\lambda_{n_j}}(x)|^p dx\\
&\le\liminf_{j\to+\infty}\frac{1}{\lambda_{n_j}^p}\int_{C_{n_j}} Q_B(x,|\lambda_{n_j}\nabla v^{\lambda_{n_j}}(x)|)dx+\frac{ \overline{Q}}{E_0^p} \ell^p |B|\\
&\le\lim_{j\to+\infty}\frac{1}{\lambda_{n_j}^p}\int_B Q_B(x,|\lambda_{n_j}\nabla v^{\lambda_{n_j}}(x)|)dx+\frac{ \overline{Q}}{E_0^p} \ell^p |B|\\
&=\liminf_{\lambda\to+\infty}\frac{1}{\lambda^p}\int_{B}Q_B(x,\lambda|\nabla v^{\lambda}(x)|)dx+\frac{ \overline{Q}}{E_0^p} \ell^p |B|,
\end{split}
\end{equation*}
where in the second line we exploited $\beta(x)-\delta \geq 0$ a.e. in $B$, in the fourth line we exploited $\beta\leq \overline Q / E_0^p$, as follows from (A3) and (A4) and in the last line we exploited \eqref{seq_to_liminf_lem_infty}. Therefore, the conclusion follows from the arbitrariness of $\ell>0$.
\end{proof}

We are now in a position to prove the previous result in the general case expressing the asymptotic behaviour of the Dirichlet Energy for the outer region in terms of a factorized $p-$Laplacian form. Specifically, the following Proposition provides a fundamental inequality relating the two cases.

\begin{prop}\label{fund_ine_prop}
Let $1<p<+\infty$, $f\in X^p_\diamond(\partial \Omega)$, (A1), (A2), (A3), (A4) hold and the solutions $v^\lambda$ of \eqref{G_norm} weakly converge to $v$ in $W^{1,p}(B)$, as $\lambda\to+\infty$. Then \begin{equation}
\label{fundamental_inequality3}
\int_{B}\beta(x)|\nabla v(x)|^pdx\le
\liminf_{\lambda\to+\infty}\frac 1 {\lambda^p}\int_{B} Q_B(x,\lambda|\nabla v^\lambda(x)|)dx,
\end{equation}
where $\beta$ is given in (A4).
\end{prop}
\begin{proof}
We assume that $v$ is nonconstant, otherwise the conclusion is trivial. Therefore, the integral at the l.h.s. of \eqref{fundamental_inequality3} is positive because $\beta\in L^\infty_+(B)$.

We firstly observe that the measure
\begin{equation*}
\kappa: F\in {\mathcal B(\Omega)}\mapsto \int_F\beta(x)|\nabla v(x)|^pdx,
\end{equation*}
 {where $\mathcal B(\Omega)$ is the class of the borelian sets contained in $\Omega$,}
is absolutely continuous with respect to the Lebesgue measure; then, for any $\varepsilon>0$, there exists $\theta>0$ such that for $|F|<\theta$ implies $\kappa(F)<\frac\varepsilon 2$.  {Then, for any $0<\delta\le\inf_{B}\beta$}, the set $F_{\delta, \theta}$, given in Lemma \ref{lem_dis_fund_infty}, satisfies $|B\setminus F_{\delta,\theta}|<\theta$ and implies that
\[
\kappa(B\setminus F_{\delta,\theta})=\int_{B\setminus F_{\delta,\theta}}\beta(x)|\nabla v(x)|^pdx <\frac \varepsilon 2.
\] 

Let $\{\lambda_n\}_{n\in\N}$ be an increasing sequence such that  $\lambda_n\to+\infty$ as $n\to+\infty$, with $\lambda_n>0\ \forall n\in\N$ satisfying \eqref{seq_to_liminf_lem_infty} appearing in Lemma \ref{lem_dis_fund_infty}. 
We have
\begin{equation*}
\begin{split}
&\int_{B} \beta(x)|\nabla v(x)|^p dx-\varepsilon=\int_{B\setminus F_{\delta,\theta}} \beta(x)|\nabla v(x)|^p dx+\int_{F_{\delta,\theta}} \beta(x)|\nabla v(x)|^p dx- \varepsilon\\
&<\int_{F_{\delta,\theta}} \beta(x)|\nabla v(x)|^p dx- \frac \varepsilon 2 =\int_{F_{\delta,\theta}} \beta(x)|\nabla v(x)|^p dx-\delta \int_{B} |\nabla v(x)|^p dx\\
&\le \int_{F_{\delta,\theta}} (\beta(x)-\delta)|\nabla v(x)|^p dx\le \liminf_{n\to+\infty}\int_{F_{\delta,\theta}} (\beta(x)-\delta)|\nabla v^{\lambda_n}(x)|^p dx\\
&\le\liminf_{\lambda\to+\infty}\frac{1}{\lambda^p}\int_{B}Q_B(x,t|\nabla v^{\lambda}(x)|)dx.
\end{split}
\end{equation*}
where we applied the absolute continuity of $\kappa$ in the second line, we set  {$\delta\leq\min\{\varepsilon/\left(2\int_{B}|\nabla v(x)|^p dx\right), \inf_B \beta(x)\}$}, we applied the lower semicontinuity of the integral functional in the third line, and we exploited  the \eqref{step_to_conclude_infty} of Lemma \ref{lem_dis_fund_infty} in the fourth line.

The conclusion follows from letting $\varepsilon\to 0^+$.
\end{proof}

\section{Limit for large Dirichlet data}\label{large_sec}
In this Section, we treat the limiting case of problem \eqref{G_norm} for large boundary data, i.e when $\lambda$ approaches infinity. 

We distinguish two cases according to the values of $p$ and $q$, which correspond to the  {asymptotic} growth exponents of the electrical conductivity in $B$ and $A$, respectively:
\begin{enumerate}
    \item[(i)] $1<q<p<+\infty$ (see Section \ref{Large_bigger});    
    \item[(ii)] $1<p<q<+\infty$ (see Section \ref{Large_smaller}).
\end{enumerate}

Specifically, in the first case, we prove that: 
\begin{itemize}
\item[(i.a)] $v^\lambda\rightharpoonup v_B$ in $W^{1,p}(B)$, as $\lambda\to +\infty$, where $v_B$ in $B$ is the unique solution of problem \eqref{pproblem_B};
\item[(i.b)] $v^\lambda\rightharpoonup v_A$ in $W^{1,q}(A)$, as $\lambda\to +\infty$, where $v_A$ in $A$ is the unique solution of problem \eqref{qproblem_A} (under an additional hypothesis);
\end{itemize}
For $1<q<p$, the limiting solution is characterized by the following  problems:
\begin{equation}
    \label{Hinfty}
\min_{\substack{v\in W^{1,p}(B)\\ v=f\ \text{on}\ \partial \Omega}}\mathbb B(v),\quad \mathbb B(v)=\int_{B} \beta(x)|\nabla v(x)|^pdx
,
\end{equation}
\begin{equation}
\label{Linfty}
\min_{\substack{v\in W^{1,q}(A)\\ v=v_B\ \text{on}\ \partial A}}\mathbb A (v),\quad \mathbb A(v)=\int_{ A} \alpha(x)|\nabla v(x)|^q dx.
\end{equation}
Hereafter, we  {recall that we have denoted by} $v_B\in W^{1,p}(B)$ and $v_A\in W^{1,q}(A)$ as the unique solutions of \eqref{Hinfty} and \eqref{Linfty}, respectively.
 
Problems \eqref{Hinfty} and \eqref{Linfty} are the variational form of problems \eqref{pproblem_B} and \eqref{qproblem_A}, respectively.

Whereas, in the second case, we prove that:

\begin{itemize}
\item[(ii.a)] $v^\lambda\rightharpoonup  w$ in $W^{1,p}(\Omega)$, as $\lambda\to +\infty$, where $ w$ in $B$ is the unique solution of problem \eqref{pproblem_Bgrad} and it is constant in each connected component of $A$.
\item[(ii.b)] $v^\lambda\to  w$ in $W^{1,q}(A)$, as $\lambda\to +\infty$, where $ w$ is constant in any connected component of $A$.
\end{itemize}

In the latter case, the constant value assumed by $ w$ in each connected component of $A$ guarantees the continuity of $ w$ across the boundaries with $B$.

For $1<p<q$, the limiting solution is characterized by the following problem:
\begin{equation}
\label{N}
\min_{\substack{v\in W^{1,p}(\Omega)\\ |\nabla v|=0\ \text{a.e. in}\ A \\ v=f\ \text{on}\ \partial \Omega}}\mathbb B(v),\quad \mathbb B(v)=\int_{B} \beta(x)|\nabla v(x)|^{p} dx.
\end{equation}
The unique solution of \eqref{N} is denoted by $ {w\in W^{1,p}(\Omega)}$. Problem \eqref{N} is the variational form for problem \eqref{pproblem_Bgrad}.

\begin{rem}
From the physical standpoint, the function $ {w}\in W^{1,p}(\Omega)$ solution of problem \eqref{N}, corresponds to the solution of a  {weighted}  $p-$Laplace problem in $B$ with region $A$ filled by a PEC.
\end{rem}

For the sequel, we need an upper bound for both the Dirichlet Energy density and the $L^p-$norm of the gradient of the solution $v^\lambda$ (in the outer region) of problem \eqref{G_norm}, as provided by the following Lemma.
\begin{lem}\label{primadiEtilde}
Let $1<p<+\infty$, $f\in X^p_\diamond(\partial \Omega)$, (A1), (A2) and (A3) hold and $ {w}\in W^{1,p}(\Omega)$ be the function achieving the minimum of the variational problem \eqref{N}.
Then, the solution $v^\lambda$ of \eqref{G_norm} satisfies  
\begin{equation}
\label{etildegrad}
\int_{B} |\nabla v^\lambda(x)|^p dx\leq \overline Q /\underline Q 
\left(\int_{B} |\nabla  {w} (x)|^p dx+ 
\frac{E^p_0}{\lambda^p} |B|\right)+\frac{ {E_0^p}}{\lambda^p} |B|
\end{equation}
and
\begin{equation}
\label{etildeQ}
\begin{split}
\mathbb G^\lambda (v^\lambda)
\leq
\frac{\overline Q}{E_0^p}\int_{B} |\nabla  {w} (x)|^p dx+\frac{\overline Q}{\lambda^p} |B|.
\end{split}
\end{equation}
\end{lem}
\begin{proof}
Let us consider the two auxiliary functions 
\[
\begin{split}
    Q_1( E)=\underline  Q\left[\left(\frac{E}{E_0}\right)^p- {1}\right]\quad\forall E>0,\\
    Q_2(E)=\overline  Q\left[\left(\frac{E}{E_0}\right)^p+ {1}\right]\quad\forall E>0.
\end{split}
\]
In terms of $Q_1$ and $Q_2$, the assumption (A3) reads as
\begin{equation}
\label{chainQ}
Q_1(E)\le Q_B(x,E)\le Q_2(E) \quad \text{for a.e.} \  x\in B \ \text{and}\ \forall E>0.
\end{equation}
Therefore, we have
\begin{equation}
\label{proof_prelim_lemma}
\begin{split}
\int_{B} |\nabla v^\lambda(x)|^pdx-\frac{ {E_0^p}}{\lambda^p} |B|&=\frac{E_0^p}{\underline Q}\frac{1}{\lambda^p} \int_{B}Q_1(\lambda|\nabla v^\lambda (x)|)dx\\
&\leq\frac{E_0^p}{\underline Q}\frac{1}{\lambda^p}\int_{B} Q_B(x,\lambda|\nabla v^\lambda (x)|)dx\le\frac{E_0^p}{\underline Q}\mathbb G^\lambda( v^\lambda)\\
&\le \frac{E_0^p}{\underline Q}\mathbb G^\lambda( {w})= \frac{E_0^p}{\underline Q}\frac{1}{\lambda^p}\int_{B} Q_B(x, \lambda|\nabla  {w} (x)|) dx\\
& \le \frac{E_0^p}{\underline Q}\frac{1}{\lambda^p}\int_{B} Q_2( \lambda |\nabla  {w} (x)|) dx\\
& = \frac{\overline Q}{\underline Q}\int_{B} |\nabla  {w} (x)|^p dx+\frac{\overline Q}{\underline Q}\frac{ {E_0^p}}{\lambda^p} |B|,
\end{split}
\end{equation}
where in the first line we applied the definition of $Q_1$ when $E=\lambda\nabla v^\lambda$, in the second line we exploited the leftmost inequality \eqref{chainQ}, in the third line we applied  {$w$} as the test function for functional \eqref{G_norm} and  {exploited that $|\nabla w|=0$ a.e. in $A$}, in the fourth line we exploited the rightmost inequality (\ref{chainQ}), and in the fifth line we applied the definition of $Q_2$ when $E=\lambda\nabla  {w}$.

This concludes the proof of \eqref{etildegrad}.

 {Inequality \eqref{etildeQ} is a byproduct of \eqref{proof_prelim_lemma}, obtained by comparing the last term of the second line with the final term.}
\end{proof}

\subsection{First case: $\mathbf {q <p}$}
\label{Large_bigger}

For any fixed $f\in X_\diamond^p(\partial \Omega)$, we study the problem as $\lambda$ approaches infinity. Since $q<p$, we have $W^{1,p}(\Omega)\cup W^{1,q}(\Omega)= W^{1,q}(\Omega)$ and, hence, the variational problem \eqref{G_norm} particularizes as
\begin{equation}
\label{G^tlarge}
\min_{\substack{v\in W^{1,q}(\Omega)\\ v=f\ \text{on}\ \partial \Omega}}\mathbb G^\lambda(v).
\end{equation}

Here, we first show that $v^\lambda$ is weakly convergent in $W^{1,p}(B)$ to the solution of problem \eqref{Hinfty}. Then, we prove that $v^\lambda$ is weakly convergent in $W^{1,q}(A)$ to the solution of problem \eqref{Linfty}, under an additional assumption.

\begin{thm}
\label{Thm_conv_limlarge}
Let $1<q<p<+\infty$, $f\in X^p_\diamond(\partial \Omega)$ and $v^\lambda$ be the solution of \eqref{G^tlarge}. If (A1), (A2), (A3) and (A4) hold, then
\begin{itemize}
\item[] $v^\lambda\rightharpoonup v_B$ in $W^{1,p}(B)$, as $\lambda\to +\infty$,
\end{itemize}
where $v_B\in W^{1,p}(B)$
is the unique solution of
\eqref{Hinfty}.
\end{thm}
\begin{proof}
Let  {$w\in W^{1,p}(\Omega)$ be the minimizer} of problem \eqref{N}. From \eqref{etildegrad} of Lemma \ref{primadiEtilde}, we have
\begin{equation}
\label{limitatezza_v_lambda}
\int_{B} |\nabla v^\lambda(x)|^p dx\leq{\overline Q}/{\underline Q}\left(\int_{B} | {w} (x)|^p dx+\frac{ {E_0^p}}{\lambda^p} |B|\right)+\frac{ {E_0^p}}{\lambda^p} |B|<\infty.
\end{equation}
Since the last two integral terms tend to $0$ as $\lambda\to+\infty$, then $||\nabla v^\lambda||_{L^p(B)}^p$ is definitively upper bounded. Moreover, up to subsequences, since $v^\lambda=f$ on $\partial\Omega$, then \cite[Chap. 11]{leoni17} there exists $v^\infty\in W^{1,p}(B)$, with $v^\infty=f$, such that
\begin{equation}
\label{conv_vt_v_infty}
v^\lambda \rightharpoonup v^\infty\quad\text{in}\ W^{1,p}(B)\quad\text{as}\ \lambda\to+\infty.
\end{equation}

Let $\tilde v_B\in W^{1,p}(\Omega)$ be an extension of $v_B\in W^{1,p}(B)$. By using $\tilde v_B$ as the test function in \eqref{G^tlarge}, and from the rightmost inequality of assumption ( {A3}.ii), we have
\begin{equation}
    \label{duepezzi}
\mathbb{G}^\lambda(v^\lambda)\leq \mathbb{G}^\lambda(\tilde v_B)\leq \frac 1 {\lambda^p}\int_{B} Q_B(x,\lambda|\nabla v_B(x)|)dx+\frac{ \overline{Q}}{E_0^q}\frac 1 {\lambda^{p-q}}\int_{A} |\nabla \tilde v_B(x)|^qdx+\frac{ \overline{Q}} {\lambda^{p}}|A|.
\end{equation}
By using the Dominated Convergence Theorem together with (A3) and (A4) to treat the first term on the r.h.s. of \eqref{duepezzi}, and by taking into account that the second and third terms vanish, we have
\begin{equation}\label{caso3convupii}
\limsup_{\lambda\to+\infty}\mathbb{G}^\lambda(v^\lambda)\le\int_{B}\beta(x)|\nabla v_B(x)|^p dx=\mathbb B(v_B).
\end{equation}

Since $v^\lambda$ is weakly convergent (see \eqref{conv_vt_v_infty}), we can apply Proposition \ref{fund_ine_prop} which gives:
\begin{equation}
\label{caso3convdownii}
\mathbb B(v^\infty)\leq \liminf_{\lambda\to +\infty}\frac{1}{\lambda^p}\int_{B}Q_B(x,|\nabla v^\lambda(x)|)dx.
\end{equation}

Therefore, we have the following inequalities
\begin{equation}\label{chain_large_i}
    \begin{split}
\mathbb B(v_B)\leq \mathbb B(v^\infty)&\leq \liminf_{\lambda\to +\infty}\frac{1}{\lambda^p}\int_{B}Q_B(x,\lambda|\nabla v^\lambda(x)|)dx\\
&\leq\limsup_{\lambda\to+\infty}\mathbb{G}^\lambda(v^\lambda)\leq\lim_{\lambda\to+\infty}\mathbb{G}^\lambda(\tilde v_B)= \mathbb B(v_B),
    \end{split}
\end{equation}
where in the first inequality we have used $v^\infty$ as test function in $\mathbb B$, the second inequality is \eqref{caso3convdownii}, the third inequality is obtained by adding the integral of the Dirichlet Energy density in $A$, the fourth inequality  is obtained by using $\tilde v_B$ as the test function in $\mathbb G^\lambda$, and the last inequality is \eqref{caso3convupii}.

This implies that $\mathbb B(v_B)= \mathbb B(v^\infty)$ and $v_B=v^\infty$, because of the uniqueness of the solution of problem \eqref{Hinfty}. This result together with \eqref{conv_vt_v_infty}, gives the conclusion.
\end{proof}

\begin{rem}
From \eqref{chain_large_i}, we find that the fundamental inequality \eqref{caso3convdownii}, also stated in Proposition \ref{fund_ine_prop}, holds as an equality:
\begin{equation*}
\mathbb B(v^\infty)=\lim_{\lambda\to +\infty}\frac{1}{\lambda^p}\int_{B}Q_B(x,|\nabla v^\lambda(x)|)dx.
\end{equation*}
\end{rem}

The boundedness of $v^\lambda$ in $W^{1,p}(B)$ implies the boundedness of $Tr(v^\lambda)\in X^p(\partial A)$ and therefore $Tr(v^\lambda)\in X^q(\partial A)$. Then, by the Inverse Trace inequality in Besov spaces \cite[Th. 18.40]{leoni17}, we have the boundedness of $v^\lambda\in W^{1,q}(A)$. Then, up to a subsequence, there exists $v_A\in W^{1,q}(A)$ such that  $v^\lambda\rightharpoonup v_A$ in $W^{1,q}(A)$.

Indeed, we recall from Theorem \ref{Thm_conv_limlarge} that $v^\lambda\rightharpoonup v_B$ in $W^{1,p}(B)$ as $\lambda\to +\infty$; this implies that $v^\lambda\rightharpoonup v_B$ in $W^{1,q}(B)$. Therefore,  by \cite[Cor. 18.4]{leoni17}, we have $Tr(v^\lambda)\to Tr(v_B)=Tr(v_A)$ in $L^q(\partial A)$ as $\lambda\to +\infty$. 

However, to prove the desired result, we also need the stronger convergence
\[
||Tr(v^\lambda)-Tr(v_A)||_{X^q(\partial A)}\to 0^+,\quad \text{as }\ \lambda\to+\infty.
\]
This convergence holds when $\nabla v^\lambda\to\nabla v_B$ in $L^p(B)$, by using \cite[Th. 18.40]{leoni17}. This problem will be addressed in a forthcoming paper. Anyway, it seems interesting to state the following conditional results.

\begin{lem}
\label{lpconv2}
Let $1<q<p<+\infty$, if (A1), (A2), (A3) and (A4) 
hold, and $||Tr(v^\lambda)-Tr(v_A)||_{X^q(\partial A)}\to 0^+$, as $\lambda\to+\infty$, then there exist two positive constants $C$ and $\lambda_0$ such that $||\nabla v^\lambda||_{L^q(A)}\leq C$ for any $\lambda>\lambda_0$.
\end{lem}
\begin{proof}
For any $v\in v^\lambda+W^{1,q}_0(A)$, we have
\begin{equation}
\label{chain_p_en_i}
\begin{split}
\frac{\underline{Q}}{E_0^{ {q}}}\int_A |\nabla v^\lambda(x)|^q \text{d}x {\ - \frac{\ \underline{Q} |A|}{\lambda^q}}&\leq\frac{1}{\lambda^q}\int_A Q_A(x, {\lambda}|\nabla{v^\lambda}(x)|)\text{d}x\\
&\leq\frac{\overline{Q}}{E_0^{ {q}}} \int_A |\nabla v(x)|^q\text{d}x {\ + \frac{\ \overline{Q} |A|}{\lambda^q}}
\end{split}
\end{equation}
where in the first inequality we used assumption (A3), and in the second (A3) and the minimality of $v^\lambda$.

The Inverse Trace inequality in Besov spaces \cite[Th. 18.40]{leoni17} assures the existence of a function $g\in v^\lambda+ W^{1,q}_0(A)$ such that $||\nabla g||_{q}\leq K(q,\Omega) ||Tr(v^\lambda)||_{X^q(\partial A)}$. Therefore, from \eqref{chain_p_en_i} with $v=g$ and since $Tr(g)=Tr(v^\lambda)$ on $\partial A$, we have
\begin{equation}
\label{upp_vlambda}
||\nabla v^\lambda||^q_{q}
\leq  {\frac{\overline Q}{\underline Q}}||\nabla g||^q_{q} {\ +\ \frac{(\overline Q+\underline{Q})E_0^q |A|}{\underline Q \lambda^q}}
\le  {\frac{\overline Q}{\underline Q }K}||Tr(v^\lambda)||^q_{X^q(\partial A)} {\ +\ \frac{(\overline Q+\underline{Q})E_0^q |A|}{\underline Q \lambda^q}}.
\end{equation}
Hence, $||\nabla v^\lambda||^q_{L^q(A)}$ is definitively upper bounded because of the convergence of $Tr(v^\lambda)$.

\end{proof}


In order to identify $v_A$,  {before proving the inner convergence result, we introduce the following functional
\begin{equation}
\label{Ilambda_corpo}
\mathbb I^\lambda(v)=\frac{1}{\lambda^q}\int_A Q_A(x,\lambda|\nabla v(x)|)dx,
\end{equation}
and we observe that, for any $\lambda>0$, $v^\lambda$ is also the minimizer of
\[
\min_{\substack{v\in W^{1,q}(A)\\ v=v^\lambda \ \text{on } \partial A}}\mathbb I^\lambda(v).
\]}

\begin{thm}\label{thm_Large_inner}
Let $1<q<p<+\infty$, $f\in X^p_\diamond(\partial \Omega)$ and $v^\lambda$ be the solution of \eqref{G^tlarge}. If (A1), (A2), (A3),  (A4) and (A4') 
hold, and $||Tr(v^\lambda)-Tr(v_A)||_{X^q(\partial A)}\to 0^+$, as $\lambda\to+\infty$, then
\begin{itemize}
\item[] $v^\lambda {\rightharpoonup} v_A$ in $W^{1,q}(A)$, as $\lambda\to +\infty$,
\end{itemize}
where $v_A\in W^{1,q}(A)$
is the unique solution of
\eqref{Linfty}.
\end{thm}
\begin{proof}
The Inverse Trace inequality \cite[Th. 18.34]{leoni17} assures the existence of a function $w^\lambda\in Tr(v^\lambda-v_A)+ W^{1,q}_0(A)$ such that $||\nabla w^\lambda ||_{q}\leq K(q,\Omega) ||Tr(v^\lambda)-Tr(v_A)||_{X^q(\partial A)}$. Since, by assumption, the last quantity is infinitesimal, then $||\nabla w^\lambda ||_{q}\to 0$.

 {
For any $0<t<1$, let us consider the convex combination
\[
v_A+w^\lambda=t \frac{v_A
}{t}+(1-t)\frac{w^\lambda}{1-t};
\]
then, by the convexity of $\mathbb I^\lambda$, we have
\begin{equation}
\label{convex_comb_I}
\begin{split}
\mathbb I^\lambda(v_A+w^\lambda)&\le t\mathbb I^\lambda\left( \frac{v_A}{t}\right)+(1-t)\mathbb I^\lambda\left(\frac{w^\lambda}{1-t}\right)\\
&\leq t\mathbb I^\lambda\left( \frac{v_A}{t}\right)+(1-t)\overline Q\left(\int_A\frac{|\nabla w^\lambda(x)|^q}{(1-t)^qE_0^q}dx+\frac{|A|}{\lambda^q}\right).
\end{split}
\end{equation}
by assumption (A3.ii)
.}
 {By passing to the $\liminf$ in \eqref{convex_comb_I} and taking into account that $||\nabla w^\lambda ||_{q}\to 0$ as $\lambda\to+\infty$, we have
\[
\liminf_{\lambda\to +\infty} \mathbb I^\lambda(v_A
+w^\lambda)\le t\liminf_{\lambda\to + \infty}
\mathbb I^\lambda\left( \frac{v_A
}{t}\right)=t^{1-q}\liminf_{\lambda\to + \infty}
\mathbb I^\lambda\left(
v_A\right).
\]
Therefore, the limit as $t\to 1^-$ provides
\begin{equation}
\label{IiwIi}
\liminf_{\lambda\to +\infty} \mathbb I^\lambda(v_A
+w^\lambda)\le \liminf_{\lambda\to + \infty}
\mathbb I^\lambda\left(v_A
\right).
\end{equation}}


 {From Proposition \ref{lpconv2}, we know that $||\nabla v^\lambda||_{L^q(A)}$ is upper bounded and hence there exists $v^\infty \in W^{1,q}(A)$ such that $v^\lambda\rightharpoonup v^\infty$ in $ W^{1,q}(A)$.}

Hence, we have the following inequalities
\begin{equation}
\label{conv_vt_v_infty_A}
    \begin{split}
\mathbb A(v_A)\leq \mathbb A(v^\infty)&\leq \liminf_{\lambda\to+\infty}\mathbb{I}^\lambda(v^\lambda)\leq\liminf_{\lambda\to +\infty}\mathbb{I}^\lambda(v_A
+w^\lambda)\\
&\leq\liminf_{\lambda\to +\infty}\mathbb{I}^\lambda(v_A
)=\lim_{\lambda\to+\infty}\mathbb{I}^\lambda(v_A)= \mathbb A(v_A).
\end{split}
\end{equation}
where the first inequality follows from testing $v^\infty$ in $\mathbb A$, the second inequality follows from using the fundamental result of Proposition \ref{fund_ine_prop}, the third inequality accounts the minimality of $v^\lambda$ for $\mathbb I^\lambda$, the fourth inequality is \eqref{IiwIi}, 
and the last equalities follows from using the dominated convergence theorem together with (A3) and (A4'). Proposition \ref{fund_ine_prop} was applied in the second inequality for $B$, $p$, $\beta$ and (A4) replaced by $A$, $q$, $\alpha$ and (A4'), respectively.

Equation \eqref{conv_vt_v_infty_A} implies that $\mathbb A(v_A)= \mathbb A(v^\infty)$ and $v_A=v^\infty$, by the uniqueness of the solution of problem \eqref{Hinfty}.
\end{proof}

As a corollary, we have a global convergence result.

\begin{cor}
Let $1<q<p<+\infty$, $f\in X^p_\diamond(\partial \Omega)$ and $v^\lambda$ be the solution of \eqref{G^tlarge}. If (A1), (A2), (A3), (A4) and (A4')
hold, and $||Tr(v^\lambda)-v_A||_{X^q(\partial A)}\to 0^+$, as $\lambda\to+\infty$, then
\begin{itemize}
\item[] $v^\lambda\rightharpoonup v_\Omega$ in $W^{1,q}(\Omega)$, as $\lambda\to +\infty$,
\end{itemize}
where $v_\Omega\in W^{1,q}(\Omega)$
is defined by
\begin{equation*}
v_\Omega=
\begin{cases}
v_B\quad\text{in }B\\
v_A\quad\text{in }A,
\end{cases}
\end{equation*}
$v_B\in W^{1,p}(B)$
is the unique solution of
\eqref{Hinfty}, $v_A\in W^{1,q}(A)$
is the unique solution of
\eqref{Linfty}.
\end{cor}

\subsection{Second case: $\mathbf{p<q}$
}
\label{Large_smaller}
For any fixed $f\in X_\diamond^p(\partial \Omega)$, since $p<q$, we have $W^{1,p}(\Omega)\cup W^{1,q}(\Omega)= W^{1,p}(\Omega)$ and, hence, the variational problem \eqref{G_norm} particularizes as
\begin{equation}\label{G^tlargeii}
\min_{\substack{v\in W^{1,p}(\Omega)\\ v=f\ \text{on}\ \partial \Omega}}\mathbb G^\lambda(v)
.
\end{equation}
We denote by $ {w}$ the solution of the limiting problem defined in \eqref{N}. Problem \eqref{N} is relevant because $v^\lambda$  {weakly converges to $w$} for $\lambda \to +\infty$.

We observe that the condition $|\nabla v|=0$ in \eqref{N} is equivalent to saying that $v$ is constant on each connected component of $A$. In other words, region $A$ behaves as a perfect electric conductor and, in addition, region $B$ corresponds to a $p-$Laplacian modelled material.

\begin{thm}\label{Thm_conv_limlargeii}
Let $1<p<q<+\infty$, $f\in X^p_\diamond(\partial \Omega)$ and $v^\lambda$ be the solution of \eqref{G^tlargeii}. If (A1), (A2), (A3) and (A4) hold, then
\begin{itemize}
\item[
] $v^\lambda\rightharpoonup  {w}$ in $W^{1,p}(\Omega)$, as $\lambda\to +\infty$,
\end{itemize}
where $ {w}\in W^{1,p}(\Omega)$
is the unique solution of
\eqref{N}. 
Moreover, we have 
\begin{itemize}
    \item[] $v^\lambda\to  {w}$ in $W^{1,q}(A)$, as $\lambda\to +\infty$.
\end{itemize}
\end{thm}
\begin{proof} For the sake of simplicity, we will only treat the case when $A$ has one connected component. 

Let $ {w}\in W^{1,p}(\Omega)$ the solution of \eqref{N}. By \eqref{etildegrad} of Lemma \ref{primadiEtilde}, we have
\begin{equation}
\label{chainGtvtwOmega4}
\int_{B} |\nabla v^\lambda(x)|^p dx\leq\frac{\overline Q}{\underline Q}\left(\int_{B} |\nabla  {w} (x)|^p dx+\frac{ {E_0^p}}{\lambda^p} |B|\right)+\frac{ {E_0^p}}{\lambda^p} |B|.
\end{equation}
Hence the left hand side in \eqref{chainGtvtwOmega4} is upper bounded. 
Moreover, we have
\begin{equation}
\label{chainGtvtwA}
\begin{split}
\lambda^{q-p}\int_{A}|\nabla v^\lambda(x)|^qdx  &\leq\frac{ {E_0^q}}{\lambda^p} |A|+\frac{E_0^q}{\underline Q \lambda^p}\int_{A} Q_A(x,|\lambda\nabla v^\lambda(x)|)dx\\
&\leq\frac{ {E_0^q}}{\lambda^p} |A|+\frac{E_0^q}{\underline Q} \mathbb G^\lambda(v^\lambda)\\
&\leq\frac{ {E_0^q}}{\lambda^p} |A|+\frac{\overline{Q}}{\underline Q}E_0^{q-p}\int_{B} |\nabla  {w}(x)|^pdx+\frac{ \overline{Q}}{\underline Q }\frac{{  {E_0^q}}}{\lambda^{ {p}}} |B|.
    \end{split}
\end{equation}
where in the first inequality we used the left-hand side of assumption (A3.ii), in the second inequality we used the definition of $\mathbb G^\lambda$ and in the third inequality we exploited \eqref{etildeQ} of Lemma \ref{primadiEtilde}.

From \eqref{chainGtvtwOmega4} and \eqref{chainGtvtwA}, it is clear that $v^\lambda\in W^{1,p}(\Omega)\cap W^{1,q}(A)$ and that $||v^\lambda||_{W^{1,p}(\Omega)}$ is upper bounded. 
Therefore, taking into account that $v^\lambda=f$ on $\partial\Omega$ for any $\lambda>0$, 
there exists a function $v^\infty\in W^{1,p}(\Omega)$ such that $v^\lambda\rightharpoonup v^\infty$ in $W^{1,p}(\Omega)$, up to a subsequence, with $v^\infty=f$ on $\partial\Omega$.

Finally, from \eqref{chainGtvtwA}, we find that 
$\int_{A}|\nabla v^\lambda(x)|^qdx=O(\lambda^{p-q})$. 
Therefore, $v^\lambda\to v^\infty$ in $W^{1, {q}}(A)$ and $v^\infty$ is constant in $A$ because $\nabla v^\lambda \to 0$ in  {$L^{q}(A)$}.

Since assumption (A4) holds, then we are in position to apply Proposition \ref{fund_ine_prop} which gives
\begin{equation}
\label{fundamental_inequality_inftyii}
\mathbb B(v^\infty)\le
\liminf_{\lambda\to +\infty}\frac 1 {\lambda^p}\int_{B} Q_B(x,\lambda|\nabla v^\lambda(x)|)dx.
\end{equation}

Consequently, we have the following inequalities
\begin{equation}\label{chain_large_ii}
    \begin{split}
\mathbb B(v_\Omega)\le\mathbb B(v^\infty)&\le\liminf_{\lambda\to +\infty}\frac 1 {\lambda^p}\int_{B} Q_B(x,|\lambda\nabla v^\lambda(x)|)dx\\
&\le \liminf_{\lambda\to +\infty}\mathbb G^\lambda(v^\lambda) \le \lim_{\lambda\to +\infty}\mathbb{G}^\lambda( {w})=\ \mathbb B( {w}),
\end{split}
\end{equation}
where in the first inequality we used $v^\infty$ as the test function in $\mathbb B$, in the second inequality we used \eqref{fundamental_inequality_inftyii}, in the third inequality we used the definition of $\mathbb G^\lambda$, in the fourth inequality we used $w$ as the test function in $\mathbb G^\lambda$, and in the  {equality} we used assumption (A4) and the dominated convergence Theorem.

This implies that $\mathbb B( {w})=\mathbb B(v^\infty)$ and hence $v^\infty= {w}$ by the uniqueness of the solution of problem \eqref{N}.
\end{proof}

\begin{rem}
From \eqref{chain_large_ii}, we find that the fundamental inequality \eqref{fundamental_inequality_inftyii}, also stated in Proposition \ref{fund_ine_prop}, holds as an equality:
\[
\mathbb B(v^\infty)=
\liminf_{\lambda\to +\infty}\frac 1 {\lambda^p}\int_{B} Q_B(x,\lambda|\nabla v^\lambda(x)|)dx.
\]
\end{rem}

\section{The pointwise convergence assumptions in the limiting case}
\label{counter_sec}
The main aim of this Section is to prove that assumption (A4) for large Dirichlet data is sharp. Specifically, we provide a counterexample where (A4) does not hold, while (A1), (A2) and (A3) still hold. For this case, we prove that, for the Dirichlet energy \eqref{G_norm}, the convergence results of Theorems \ref{Thm_conv_limlarge} 
and \ref{Thm_conv_limlargeii} (and hence of Theorem \ref{thm_Large_inner}) do not hold.

 {With a similar approach, not reported here for the sake of brevity, it is possible to prove that even assumption (A4') is sharp.}

In order to prove this result, we need to build a Dirichlet energy function $Q_B$ such that the ratio $Q_B(x,E)/E^p$ does not admit the limit for $E \to +\infty$. This is the case when the ratio $Q_B(x,E)/E^p$ oscillates between two different values. In the following Lemma we prove that a Dirichlet Energy density of this type exists (see Figure \ref{fig_8_counterlarge} for the geometrical interpretation).

\begin{figure}
    \centering
    \includegraphics[width=0.6\textwidth]{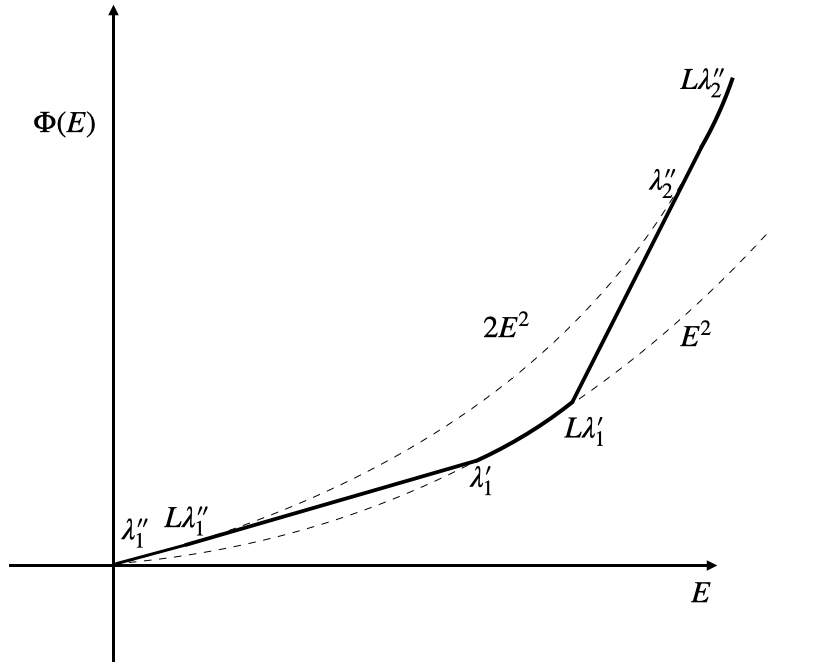}
    \caption{The continuous line represents the function $\Phi$ for the counter example. It gives the Dirichlet Energy density as $\Psi(E)=\Phi(E)+E^2$.}
    \label{fig_8_counterlarge}
\end{figure}

\begin{lem}
\label{succ_Llarge}
Let $L>1$, then there exist two sequences
\[
\lambda_n'\uparrow +\infty \quad\text{and}\quad \lambda_n''\uparrow +\infty
\]
such that
\[
L \lambda''_n<\lambda'_{n},\ \  L\lambda'_{n}<\lambda''_{n+1}
\quad\forall n\in\N,
\]
and a strictly convex function $\Psi:[0,+\infty[\to[0,+\infty[$  such that
\[
\Psi|_{[\lambda_n',L \lambda_n']}(E)=2E^2\quad \Psi|_{[\lambda_n'',L \lambda_n'']}(E)=3E^2.
\]
\end{lem}
\begin{proof}
Let us fix $\lambda_1''>0$.
For each $n \in \mathbb{N}$ we set the auxiliary function $\Phi$ equal to $2 E^2$ in $(\lambda_n'', L \lambda_n'')$ and equal to $E^2$ in $(\lambda_n', L \lambda_n')$. In interval $(L \lambda_n'',\lambda_n')$ the function $\Phi$ is equal to the tangent line to function $2E^2$ evaluated at $L \lambda_n''$. 
Point $\lambda_n'$ is found at the intersection of this tangent line with function $E^2$. In interval $(L \lambda_n',\lambda_{n+1}'')$ the function $\Phi$ is a straight line, continuous at $L \lambda_n'$ and tangent to $2E^2$.
Point $\lambda_{n+1}''$ is found as the abscissa of the tangent point between this straight line and function $2E^2$. This procedure is applied iteratively from $n=1$. Function $\Phi$ turns out to be convex and sequences $\{\lambda_n'\}_{n \in \mathbb{N}}$ and $\{\lambda_n''\}_{n \in \mathbb{N}}$ are monotonically increasing to infinity.

Therefore, the measure of intervals where $\Phi$ is equal to $E^2$ or equal to $2E^2$ is nonvanishing.
It is possible to prove that the provided sequences $\{\lambda'_n\}_{n\in\N}$ and $\{\lambda_n''\}_{n\in\N}$ are two geometrical progressions. Indeed
\[\lambda_n'=c_1L\lambda_n'',\ \lambda_{n+1}''=c_2L \lambda_n',\ \lambda_{n+1}'=C^2L^2\lambda_n',\ \lambda_{n+1}''=C^2L^2\lambda_n'',
\]
where $c_1=2+\sqrt 2$, $c_2=1+\frac{\sqrt 2}{2}$ and $C$ is the geometrical average of $c_1$ and $c_2$, that is $C^2
=\left(3+2\sqrt 2\right)$.

Finally, we set $\Psi(E)=\Phi(E)+E^2$. $\Psi$ is a strictly convex function.
\end{proof}
The Dirichlet energy density defined as $Q_B(x,E)=\Psi(E)$, satisfies all the assumptions except (A4). This energy density is the basis to build a counterexample proving that (A4) is sharp. Specifically, we consider a 2D case ($n=2$) and $p=2$ in the outer region. The  {asymptotic} growth exponent $q$ satisfies condition $1<q<\infty$.

Let $r$ be greater than or equal to 10, and let the outer region $B$ be the annulus centred in the origin with radii $1$ and $r$. This annulus is $D_r\setminus\overline D_1$,  {where $D_r$ and $D_1$ are the disks of radii $r$ and $1$, respectively,} and centered at the origin. The inner region is, therefore, $D_1$. We focus on problem \eqref{G_norm}, where the Dirichlet energy density is defined as
\[
\begin{split}
Q_B(x,E)&=\Psi(E)\quad\text{in}\ D_r\setminus\overline D_1\times [0,+\infty[,\qquad\\ Q_A(x,E)&=E^q\qquad\text{in}\ D_1\times [0,+\infty[.
\end{split}
\]
Let $\gamma$ be defined as $\gamma=7+\frac {12}{r^2}$. We denote $x=(x_1,x_2) \in \mathbb{R}^2$ and we consider the problem
\begin{equation}
\label{G^tc1large}
\min_{\substack{v\in W^{1,q}(D_r)\\ v=\gamma x_1\ \text{on}\ \partial D_r}}\mathbb G^\lambda(v),\quad \mathbb G^\lambda(v)=\frac 1 {\lambda^2}\left(\int_{ D_r\setminus D_1} \Psi(\lambda|\nabla v(x)|)dx+\int_{D_1} \lambda^q|\nabla v(x)|^q dx\right).
\end{equation}

Here we prove that $\lim_{\lambda\to +\infty}\mathbb G^\lambda(v^\lambda)$ does not exist. Specifically, the two sequences $\{\lambda_n'\}_{n\in\N}\uparrow +\infty$ and $\{\lambda_n''\}_{n\in\N}\uparrow+\infty$ of Lemma \ref{succ_Llarge} give 
\begin{equation}
   \label{counter_resultlarge}
\limsup_{n\to +\infty}\mathbb{G}^{\lambda_n'}(v^{\lambda_n'})\le\ell_1<\ell_2\le\liminf_{n\to +\infty}\mathbb{G}^{\lambda_n''}(v^{\lambda_n''}).
\end{equation}
As usual, $v^\lambda$ is the solution of \eqref{G^tc1large}.

Let us consider the following problem
\begin{align}
\label{problem_down}
&\min_{\substack{v\in H^{1}(D_r\setminus\overline D_1)\\ v=\gamma x_1\ \text{on}\ \partial D_r\\v=const.\ \text{on}\ \partial D_1}}\mathbb C (v), \quad \mathbb C(v)=\int_{ D_r\setminus D_1} |\nabla v(x)|^2dx.
\end{align}
The symmetry of the domain and the zero average of the boundary data imply that the constant appearing in \eqref{problem_down} on $\partial D_1$ is zero.

An easy computation reveals that
\[
v_{D_r}(x)=\frac{7r^2+12}{r^2-1}\left(1-\frac{1}{x_1^2+x_2^2}\right)x_1\quad\text{in}\ D_r\setminus\overline D_1
\]
is the solution of 
\eqref{problem_down}, that $\Delta v_{D_r}=0$ in $D_r\setminus\overline D_1$, and that we have
\[
\begin{split}
\frac{7r^2-12}{r^2-1}\left(1-\frac{1}{\rho^2}\right)\le|\nabla v_{D_r}(x)|\le\frac{7r^2+12}{r^2-1} \left(1+\frac 1{\rho^2}\right)\quad\text{on}\ \partial D_{\rho},\ 1<\rho\le r.
\end{split}
\]
Consequently, when $\rho\ge2$, we have
\begin{equation}
    \label{stima_nablaw}
1\le\frac 34 \frac{7r^2-12}{r^2-1}\le|\nabla v_{D_r}(x)|\le\frac 54 \frac{7r^2+12}{r^2-1}\leq 10\quad\text{in}\ D_r\setminus D_2.
\end{equation}

Let $L$ be greater than 10, $\lambda_n'\uparrow +\infty$ and let $\lambda_n''\uparrow +\infty$ be the  two sequences of Lemma \ref{succ_Llarge}.  
It turns out that
\begin{equation}
\label{stime_up_down}
\lambda_n' \le
\lambda_n'|\nabla v_{D_r}(x)|
\le L \lambda'_n\quad\text{in}\ D_r\setminus D_2.
\end{equation}
We have
\begin{equation}
    \label{inf_counter}
\begin{split}
&\limsup_{n\to +\infty}\mathbb G^{\lambda'_n}(v^{\lambda_n'})\leq\limsup_{n\to +\infty} \mathbb G^{\lambda'_n}(v_{D_r})\\
&=\limsup_{n\to +\infty}\frac 1{(\lambda'_n)^2} \int_{D_r\setminus D_2}\Psi(\lambda'_n|\nabla v_{D_r}(x)|)dx+\limsup_{n\to +\infty}\frac 1{(\lambda'_n)^2} \int_{D_2\setminus D_1}\Psi(\lambda'_n|\nabla v_{D_r}(x)|)dx\\
&\leq 2 \int_{D_r\setminus D_2}|\nabla v_{D_r}(x)|^2dx+3 \int_{D_2\setminus D_1}|\nabla v_{D_r}(x)|^2dx,
\end{split}
\end{equation}
where in the first line we used the minimality of $v^{\lambda'_n}$ for $\mathbb G^{\lambda'_n}$ and that $v_{D_r}$ is an admissible function for problem \eqref{G^tc1large},
in the second line we exploited the property that the gradient of $v_{D_r}$ in $D_1$ is vanishing, and in the third line we used \eqref{stime_up_down}. By setting $\ell_1$ equal to \eqref{inf_counter}:
\[
\ell_1:=2 \int_{D_r\setminus D_2}|\nabla v_{D_r}(x)|^2dx+3 \int_{D_2\setminus D_1}|\nabla v_{D_r}(x)|^2dx,\]
we have the leftmost inequality in \eqref{counter_resultlarge}.

To obtain the rightmost inequality in \eqref{counter_resultlarge}, we consider the following problems
\begin{equation}
    \label{AuxF}
\min_{\substack{v\in H^{1}(D_r\setminus\overline D_1)\\ v=\gamma x_1\ \text{on}\ \partial D_r}} \mathbb H^\lambda (v),\quad\mathbb H^\lambda(v)= \frac{1}{\lambda^2}\int_{D_r\setminus D_2}\Psi(\lambda|\nabla v(x)|)dx+2 \int_{D_2\setminus D_1}|\nabla v(x)|^2dx.
\end{equation}
\begin{equation}
    \label{problem_up}
\min_{\substack{v\in H^{1}(D_r\setminus\overline D_1)\\ v=\gamma x_1\ \text{on}\ \partial D_r}}\mathbb D(v),\quad \mathbb D (v) =3\int_{D_r\setminus D_2}|\nabla v(x)|^2dx+2\int_{D_2\setminus D_1}|\nabla v(x)|^2dx.
\end{equation}
The unique solution of \eqref{problem_up} is
\[
w_{D_r}(x)=
\begin{cases}
\left(7+\frac{12}{x_1^2+x_2^2}\right)x_1&\quad\text{in}\ D_r\setminus\overline D_2\\
8\left(1+\frac{1}{x_1^2+x_2^2}\right)x_1 &\quad\text{in}\ D_2\setminus\overline D_1.
\end{cases}
\]
Analogously to \eqref{stima_nablaw}, it can be easily proved that
\[
1\le 4\le \left(7-\frac{12}{\rho^2}\right)\le|\nabla w_{D_r}(x)|\le \left(7+\frac {12}{\rho^2}\right)\le 10< L\quad\text{on}\ \partial D_{\rho},\ 2\le\rho\le r.
\]
and hence we choose $L>10$ such that
\begin{equation}
\label{stime_up}
\lambda_n'' \le
\lambda_n''|\nabla w_{D_r}(x)|
\le L \lambda''_n\quad\text{in}\ D_r\setminus D_2.
\end{equation}
Therefore, we have
\begin{equation}
\label{sup_counter}
\mathbb{G}^{\lambda_n''}(v^{\lambda_n''})\geq\mathbb H^{\lambda''_n}(w_{D_r})=\mathbb D(w_{D_r}),
\end{equation}
where the inequality comes from the definition of $\Psi$. The equality follows from the fact that $\mathbb H^{\lambda''_n}$ coincides with $\mathbb D$ by \eqref{stime_up} and the definition of $\Psi$. We highlight that $w_{D_r}$ is a local minimizer in $W^{1,\infty}(D_r)\cap W^{1,2}(D_r)$, since $\Psi$ does not depend on $x$  (see \cite{cianchi2010global} for details). Finally, $w_{D_r}$ is a global minimizer thanks to the uniqueness of \eqref{AuxF}.

By setting
\[\ell_2(r):=
\mathbb D(w_{D_r}).
\]
we have the rightmost inequality in \eqref{counter_resultlarge} by passing to the limit in \eqref{sup_counter}.

At this stage, it only remains to be proved that $\ell_1(r)<\ell_2(r)$. To this purpose, we notice that:
\[
\begin{split}
 \ell_1(r)&= 2 \int_{D_r\setminus D_2}|\nabla v_{D_r}(x)|^2dx+3 \int_{D_2\setminus D_1}|\nabla v_{D_r}(x)|^2dx\\
\ell_2(r)&= 3\int_{D_r\setminus D_2}|\nabla w_{D_r}(x)|^2dx+2\int_{D_2\setminus D_1}|\nabla w_{D_r}(x)|^2dx.
\end{split}
\]
Condition $\ell_1(r)<\ell_2(r)$ holds for large $r$, by observing that (i) $v_{D_r}$ and $w_{D_r}$ solve the same associated Euler-Lagrange equation on $D_2\setminus\overline D_1$, (ii) $\nabla v_{D_r}(x)$ and $\nabla w_{D_r}(x)$ are bounded functions on the bounded domain $D_2\setminus D_1$ by \eqref{stime_up_down} and \eqref{stime_up}, respectively, and (iii) it turns out that
\[
\begin{split}
&\lim_{r\to+\infty}\frac{
\int_{ D_r\setminus D_2} |\nabla v_{D_r}(x)|^2dx}{\int_{ D_r\setminus D_2} |\nabla w_{D_r}(x)|^2dx}= 1,\\
&\lim_{r\to+\infty}\int_{D_r\setminus D_2}|\nabla v_{D_r}(x)|^2dx=\lim_{r\to+\infty}\int_{D_r\setminus D_2}|\nabla w_{D_r}(x)|^2dx=+\infty.
\end{split}
\]

\section{Conclusions}
\label{Con_sec}
This study is motivated by Inverse Problems in the presence of nonlinear materials, where the treatment of nonlinear constitutive relationships is still at an early stage of development, as clearly stated in \cite{lam2020consistency}. 

We focus on Electrical Resistance Tomography where the aim is to retrieve the electrical conductivity/resistivity of a material by means of stationary (DC) currents. Our main results prove that the original nonlinear problem can be replaced by a proper $p-$Laplace problem, when the prescribed Dirichlet data are \lq\lq large\rq\rq. Specifically, we prove that in the presence of two materials  {with different asymptotic growth}, the scalar potential in the outer region in contact with the boundary where the Dirichlet data is prescribed, can be computed by (i) replacing the interior region with either a Perfect Electric Conductor or a Perfect Electric Insulator and (ii) replacing the original problem (material) in the outer region with a  {weighted}  $p-$Laplace problem. In a certain sense, the presence of the \lq\lq fingerprint\rq\rq\ of a  {weighted}  $p-$Laplace problem can be recognized in an arbitrary nonlinear problem. From the perspective of tomography, this is a significant result because it highlights the central role played by the  {weighted}  $p-$Laplacian in inverse problems with nonlinear materials. For $p=2$, i.e. when the material in the outer region is linear, these results constitute a powerful bridge making it possible to bring all theoretical results, imaging methods and algorithms developed for linear materials into the arena of problems with nonlinear materials. 




The fundamental tool to prove the convergence results are the inequalities appearing in Propositions \ref{fund_ine_prop}. They express the asymptotic behaviour of the Dirichlet Energy for the outer region in terms of a factorized $p-$Laplacian form.

Moreover, we prove that our assumptions are sharp, by means of proper counterexamples. Finally,  {it would be interesting to} provide a numerical example, referring to a superconducting cable, as an application of the theoretical results proved in this paper.


\section*{Acknowledgements}
This work has been partially supported by the MiUR-Progetto Dipartimenti di eccellenza 2018-2022 grant \lq\lq Sistemi distribuiti intelligenti\rq\rq of Dipartimento di Ingegneria Elettrica e dell'Informazione \lq\lq M. Scarano\rq\rq, by the MiSE-FSC 2014-2020 grant \lq\lq SUMMa: Smart Urban Mobility Management\rq\rq\ and by GNAMPA of INdAM.



\bibliographystyle
{ieeetr}
\bibliography{biblioCFPPT}

\end{document}